\documentclass[11pt,a4paper,reqno]{article}
\usepackage[T1]{fontenc}
\usepackage[margin=2.5cm]{geometry}
\usepackage{microtype}
\usepackage{amsmath,amsfonts,amssymb,amsthm,mathrsfs}
\usepackage{mathtools}
\usepackage{booktabs}
\usepackage{graphicx}
\usepackage[colorlinks=true,linkcolor=cyan,citecolor=blue,urlcolor=blue]{hyperref}
\usepackage{enumitem}
\setlist{noitemsep,topsep=3pt}

\newtheorem{theorem}{Theorem}[section]
\newtheorem{lemma}[theorem]{Lemma}
\newtheorem{proposition}[theorem]{Proposition}
\newtheorem{corollary}[theorem]{Corollary}
\newtheorem{conjecture}[theorem]{Conjecture}
\newtheorem{question}[theorem]{Question}
\theoremstyle{definition}
\newtheorem{definition}[theorem]{Definition}

\theoremstyle{remark}
\newtheorem{remark}[theorem]{Remark}

\newcommand\U{{\mathrm U}}
\newcommand\Z{{\mathbb Z}}
\newcommand\C{{\mathbb C}}
\newcommand\R{{\mathbb R}}
\newcommand\Q{{\mathbb Q}}
\newcommand\E{{\mathbb E}}
\newcommand\Tr{{\mathrm{Tr}}}
\newcommand\Pbb{{\mathbb P}}
\newcommand\cA{{\mathsf C}}
\newcommand\Hb{{\mathsf H}}
\newcommand\Pcal{{\mathcal P}}
\newcommand\Dcal{{\mathcal D}}
\newcommand\one{{\mathbf 1}}
\newcommand\dd{{\mathrm d}}
\newcommand\ellpart{{\ell}}
\newcommand\Aut{{\mathrm{Aut}}}

\makeatletter
\newcommand{\subjclass}[2][2020]{%
\let\@oldtitle\@title%
\gdef\@title{\@oldtitle\footnotetext{#1 \emph{Mathematics subject classification.} #2}}%
}
\newcommand{\keywords}[1]{%
\let\@@oldtitle\@title%
\gdef\@title{\@@oldtitle\footnotetext{\emph{Key words and phrases.} #1.}}%
}
\makeatother

\title{Elliptic $b$-Hurwitz theory and Jack heat trace}

\author{Thibaut Lemoine\thanks{Universit\'e de Strasbourg, CNRS, UMR 7501 -- Institut de Recherche Math\'ematique Avanc\'ee, 7 rue Ren\'e Descartes, 67000 Strasbourg, France. thibaut.lemoine@math.unistra.fr}}

\hypersetup{
pdftitle={Elliptic b-Hurwitz theory and Jack heat trace},
pdfauthor={Thibaut Lemoine}
}

\begin{document}

\maketitle

\begin{abstract}
We study a deformation of the central heat trace on $\mathrm U(N)$ obtained by deforming Schur polynomials to Jack polynomials, and prove that it admits an asymptotic expansion to arbitrary order. Its coefficients are governed by elliptic $b$-Hurwitz numbers, which are genus-one counterparts of the $b$-Hurwitz theory of Chapuy and Do\l\k{e}ga \cite{ChapuyDolega22}. We construct the associated elliptic $b$-Hurwitz theory by means of generalized coverings on a torus and identify it with the genus-one closure of the genus-zero simple $b$-Hurwitz theory. At $b=0$ the construction recovers ordinary elliptic Hurwitz theory, while at $b=1$ it gives an automorphism-weighted geometric interpretation of the connected and disconnected twisted elliptic Hurwitz numbers of Hahn--Markwig \cite{HahnMarkwig26}. Our results extend the topological expansion obtained in the classical case in \cite{LemMai25,LM2} and take the form of a coupling between chiral and antichiral elliptic $b$-Hurwitz generating functions.
\end{abstract}

\tableofcontents

\section{Introduction}

The present paper is motivated by two questions which, a priori, belong to rather different parts of mathematics. On the one hand, in the realm of asymptotic spectral theory, the central heat trace on $\U(N)$ is the trace of the heat kernel on $\U(N)$ on the space of square integrable central functions, and its large-$N$ asymptotic expansion was derived in \cite{LemMai25}. It was shown later in \cite{LM2} that the expansion involved the ordinary elliptic Hurwitz theory (and even elliptic Gromov--Witten theory, thanks to Okounkov--Pandharipande's Gromov--Witten/Hurwitz correspondence \cite{OkounkovPandharipande06}). Recall that, after Weyl integration, the central Laplace--Beltrami operator $\Delta_{\U(N)}$ acts on symmetric functions of the eigenangles in the space with density $|\Delta|^2$ as the radial Laplacian $\mathscr L_N$, where $\Delta(z_1,\ldots,z_N):=\prod_{1\leq i<j\leq N}(z_i-z_j)$ is the Vandermonde determinant, and an eigenbasis is given by the rational Schur functions. We deform this picture by replacing the Weyl density with the Jack density $|\Delta|^{2/\alpha}$ for some $\alpha>0$, the Schur eigenbasis with Jack Laurent polynomials, and the quadratic Casimir spectrum with its Jack analogue. The root-system part of the resulting operator is the compact type-$A_{N-1}$ Heckman--Opdam Laplacian with root multiplicity $2/\alpha$ \cite{RemlingRosler11}, naturally acting on the determinant-one torus associated with $\mathrm{SU}(N)$. The $\U(N)$ operator $\mathscr L_{\alpha,N}$ used here extends this type-$A_{N-1}$ operator by the one-dimensional determinant direction. Table~\ref{tab:jack-deformation} summarizes this deformation, which will be developed in Section~\ref{sec:jack-heat-trace}.

\begin{table}[htbp]
\centering
\small
\setlength{\tabcolsep}{4pt}
\renewcommand{\arraystretch}{1.25}
\begin{tabular}{@{}lccc@{}}
\toprule
& $\U(N)$ picture & Radial picture & Jack deformation \\
\midrule
Hilbert space
& $L^2(\U(N))^{\mathrm{Ad}}$
& $L^2_{\mathrm{sym}}\left(\mathbb T^N,|\Delta(e^{\mathrm i\theta})|^2\dd\theta\right)$
& $L^2_{\mathrm{sym}}\left(\mathbb T^N,|\Delta(e^{\mathrm i\theta})|^{2/\alpha}\dd\theta\right)$ \\
Differential operator
& $\Delta_{\U(N)}$
& $\mathscr L_N$
& $\mathscr L_{\alpha,N}$ \\
Eigenfunctions
& $\chi_\lambda$
& $s_\lambda(e^{\mathrm i\theta})$
& $J_{\lambda,N}^{(\alpha)}(e^{\mathrm i\theta})$ \\
Spectral parameter
& $c_{2,N}(\lambda)$
& $c_{2,N}(\lambda)$
& $\kappa_{\alpha,N}(\lambda)$ \\
\bottomrule
\end{tabular}
\caption{The radial realization of the central $\U(N)$ spectral problem and its Jack deformation.}
\label{tab:jack-deformation}
\end{table}
A natural first question is therefore the following.

\begin{question}\label{ques:spectral}
Does the large-rank expansion of the central heat trace on $\U(N)$ survive this deformation?
\end{question}

On the other hand, in the realm of enumerative geometry, Hurwitz theory lies at the intersection of branched-cover geometry, symmetric-group representation theory and integrable systems. The Frobenius character formula turns covering counts into character sums, double Hurwitz generating functions form Toda $\tau$-functions \cite{Okounkov00}, and completed cycles connect Hurwitz theory with Gromov--Witten theory \cite{OkounkovPandharipande06}. Goulden and Jackson observed that the Jack deformation interpolates algebraic regimes which, at the Schur and zonal specializations, correspond to different orientability phenomena, and conjectured that the coefficients in the parameter $b=\alpha-1$ should possess a positive combinatorial interpretation in which the power of $b$ measures nonorientability \cite{GouldenJackson96}.  Chapuy and Do\l\k{e}ga subsequently turned this philosophy into a $b$-Hurwitz theory of generalized branched coverings and nonorientable constellations \cite{ChapuyDolega22}, introducing in particular the notion of ``measure of nonorientability'' (MON). This deformation retains much of the structure that makes ordinary Hurwitz theory useful, leading to many significant results in the past few years: for instance, $b$-monotone Hurwitz numbers satisfy Virasoro constraints and, at the zonal specialization, a BKP hierarchy \cite{BonzomChapuyDolega23}, while refined topological-recursion phenomena have also been established \cite{ChidambaramDolegaOsuga26}.  Recent work of Fesler--Hahn--Karev--Markwig develops a generic-$b$ cut-and-join, Fock-space and tropical theory for double $b$-Hurwitz numbers \cite{FeslerHahnKarevMarkwig25}. All these works concern the deformation of the Hurwitz theory with target a sphere. The natural question is thus:

\begin{question}\label{ques:enumerative}
Is there an elliptic counterpart of the $b$-Hurwitz theory, where the base space is a torus instead of a sphere?
\end{question}

The elliptic target is the first place where this question is both especially natural and genuinely new. Classical elliptic Hurwitz generating series exhibit modular and quasimodular structures \cite{Dijkgraaf95,BlochOkounkov00,HahnIttersumLeid22}, and the enumeration of torus covers enters the geometry and dynamics of moduli spaces of Abelian differentials \cite{EskinOkounkov01}.  At the same time, genus one is characterized algebraically by a canonical trace operation: an oriented torus is obtained by identifying the two boundary circles of a cylinder, and in a two-dimensional Frobenius theory the corresponding amplitude is the trace of the cylinder propagator. There is already strong evidence for this construction: at $b=0$ it must reproduce ordinary elliptic Hurwitz theory, whereas at $b=1$ it should correspond to the twisted elliptic Hurwitz theory constructed recently by Hahn--Markwig \cite{HahnMarkwig26}.

This paper provides a simultaneous answer to both Questions~\ref{ques:spectral} and~\ref{ques:enumerative}: it proves an asymptotic expansion of the Jack-deformed central heat trace, constructs an elliptic $b$-Hurwitz theory, and shows that the relation between the expansion of the central heat trace on $\U(N)$ and classical elliptic Hurwitz theory extends to the deformed case. The main message can be summarized by the following identification:
\[
\begin{array}{c}
\text{Jack deformation of}\\ \text{elliptic Hurwitz theory}
\end{array}
\quad\longleftrightarrow\quad
\begin{array}{c}
\text{stable coefficients of the}\\
\text{Jack-deformed central heat trace}.
\end{array}
\]
\paragraph{Main results.}

For any real number $\alpha>0$, a cell $\square=(i,j)$ of an integer partition (seen as a Young diagram) has Jack content $c_\alpha(\square)=\alpha(j-1)-(i-1)$. The total Jack content $\mathsf C_{\alpha}(\lambda)=\sum_{\square\in\lambda}c_{\alpha}(\square)$ of a partition is the eigenvalue of the Jack cut-and-join operator $\mathcal D_\alpha$ on the Jack polynomial $J_\lambda^{(\alpha)}$. Assume now that $\alpha=1+b$ with $b\geq0$. Throughout, Hurwitz numbers and coverings are understood in the possibly disconnected sense unless explicitly stated otherwise. We define the \emph{elliptic (or genus-one) $b$-Hurwitz number of degree $d$ with $r$ labelled simple branch values} by
\[
\Hb_{1,b}(d,r)=\sum_{\lambda\vdash d}\mathsf C_{1+b}(\lambda)^r.
\]
Section~\ref{sec:elliptic-b-hurwitz} shows that this elementary Jack-content power sum admits three equivalent genus-one realizations.

First, let the torus $E_T=\C/(\Z+iT\Z)$ carry the reflection $\iota[z]=[\overline z]$, and fix labelled branch values $\mathbf x=(x_1,\ldots,x_r)$ on one boundary component of the quotient annulus.  Denote by $\mathfrak H_{d,r}(E_T,\iota;\mathbf x)$ the degree-$d$ real elliptic generalized coverings of Definition~\ref{def:real-elliptic-cover}.  Choosing an admissible spanning cut $a$ and cutting/capping produces a Chapuy--Do\l\k{e}ga generalized covering with equal distinguished profiles together with equivariant sewing data.  Combining its MON weight with normalized local sewing factors gives the weighted count $\#_{b,a}\mathfrak H_{d,r}(E_T,\iota;\mathbf x)$.

Second, write $\Hb_{0,b}(\mu,\nu,r)$ for the (genus-zero) $b$-Hurwitz numbers with distinguished profiles $\mu,\nu\vdash d$ and $r$ simple ramifications.  Sewing the distinguished profiles contracts them with the Jack scalar product: for $\mu=(1^{m_1}2^{m_2}\cdots)$, the factor is $z_\mu(1+b)^{\ellpart(\mu)}$, where $z_\mu=\prod_{j\geq1}j^{m_j}m_j!$.  Third, on the degree-$d$ symmetric-function space $\Lambda_d$, the same closure is the trace of the Jack cut-and-join operator.  Altogether:
\begin{theorem}[See Theorem~\ref{thm:geometric-gluing} and Propositions~\ref{prop:cylinder-closure} and~\ref{prop:cut-and-join-trace}]
For every $d,r\geq0$ and $b\geq0$,
\begin{equation}\label{eq:intro-three-realizations}
\begin{aligned}
\Hb_{1,b}(d,r)=\#_{b,a}\mathfrak H_{d,r}(E_T,\iota;\mathbf x)=\sum_{\mu\vdash d}z_\mu(1+b)^{\ellpart(\mu)}\Hb_{0,b}(\mu,\mu,r)=\Tr_{\Lambda_d}\left(\mathcal D_{1+b}^{r}\right).
\end{aligned}
\end{equation}
\end{theorem}
At the zonal specialization $b=1$, this provides the geometric realization proposed by Hahn--Markwig \cite[Remark~5]{HahnMarkwig26}; see Proposition~\ref{prop:twisted-real-monodromy}.

We next turn back to Question~\ref{ques:spectral}.  Denote by
\[
Z_{\alpha,N}(t)=\Tr(e^{\frac{t}{2}\mathscr L_{\alpha,N}})
\]
the \emph{Jack heat trace}, seen as the $\alpha$-deformation of the central heat trace on $\U(N)$ discussed earlier.  Its large-rank expansion is naturally written in powers of $1/(N+\alpha-1)$ rather than $1/N$, the two scales coinciding only in the undeformed case $\alpha=1$.

\begin{theorem}[See Theorem~\ref{thm:heat-expansion}]\label{thm:heat-expansion-informal}
For every $t>0$ and $\alpha>0$, there are explicit coefficients $(a_r(\alpha,t))_{r\geq0}$ such that, for every $p\geq0$,
\begin{equation}
Z_{\alpha,N}(t)=\sum_{r=0}^p\frac{a_r(\alpha,t)}{(N+\alpha-1)^r} +O_{t,\alpha,p}\left((N+\alpha-1)^{-p-1}\right).
\end{equation}
\end{theorem}

The coefficients admit a probabilistic representation in terms of two random partitions and a discrete Gaussian charge, as explained in Section~\ref{subsec:probabilistic-representation}.  This is the direct analogue of the probabilistic representation used in \cite{LemMai25,LM2} and is the key input for the uniform control of the remainder.  The two sides of the paper meet in the final main result.  Although the preceding expansion is valid for every $\alpha>0$, its enumerative interpretation uses the specialization $\alpha=1+b\geq1$.  From now on we express the resulting formulas entirely in terms of $b$.  Set $q_t=e^{-t/2}$ and consider the generating function
\[
\mathscr H_b(q;\hbar):=\sum_{r\geq0}\sum_{d\geq0}q^d\frac{\hbar^r}{r!}\Hb_{1,b}(d,r).
\]
Let $M$ be a random integer with discrete Gaussian distribution:
\[
\Pbb(M=m)=\frac{e^{-\frac t2(1+b)m^2}}{\Theta_{1+b}(t)},\qquad\Theta_{1+b}(t):=\sum_{m\in\Z}e^{-\frac t2(1+b)m^2}.
\]
We define the formal chiral--antichiral generating series by
\[
\mathcal Z_{b,t}(x):=\Theta_{1+b}(t)\E_M\left[e^{\frac12 b(1+b)txM^2}\mathscr H_b(q_t e^{-(1+b)tMx};-tx)\mathscr H_b(q_t e^{(1+b)tMx};-tx)\right].
\]
For a formal series $A(x)=\sum_{r\geq0}a_rx^r$, write $[A]_{\leq p}(x):=\sum_{r=0}^pa_rx^r$ for its truncation through degree $p$.

\begin{theorem}[See Theorem~\ref{thm:controlled-chiral-antichiral}]
For every $t>0$, $b\geq0$ and $p\geq0$,
\begin{equation}\label{eq:controlled-master-expansion}
Z_{1+b,N}(t) =\bigl[\mathcal Z_{b,t}\bigr]_{\leq p}\left(\frac{1}{N+b}\right) +O_{t,b,p}\left((N+b)^{-p-1}\right).
\end{equation}
The remainder is locally uniform for $(t,b)\in(0,\infty)\times[0,\infty)$.
\end{theorem}

This fixed-order formulation is not an artifact of the proof: in fact, we will see in Proposition~\ref{prop:master-zero-radius} that, for every $t>0$ and $b\geq0$, the formal series $\mathcal Z_{b,t}(x)$ has radius of convergence zero.  Consequently, \eqref{eq:controlled-master-expansion} cannot be upgraded to an identity with a convergent power series in $1/(N+b)$: its natural meaning is an asymptotic identification to arbitrary fixed order.

At $\alpha=1$, the central heat trace is, up to the normalization of time and of the quadratic Casimir, the genus-one $\U(N)$ Yang--Mills partition function, and its expansion studied in \cite{LemMai25,LM2} corresponds to a particular case of gauge/string duality, see also \cite{GrossTaylor93,Dijkgraaf95,Lem26survey}.  For generic $\alpha$, we do not claim any gauge-theoretic interpretation in this paper.

\section{The Jack-deformed central heat trace}\label{sec:jack-heat-trace}

In this section, we introduce the Jack deformation of the central heat trace on $\U(N)$ and prove its large-$N$ expansion using the random-partition techniques developed in \cite{LemMai25,LM2}.  A stable representation of the highest weights of $\U(N)$ expresses the deformed heat trace as the expectation of an explicit functional of two random partitions and a random integer; deviation bounds then yield an asymptotic expansion with a controlled remainder.

\subsection{The deformed radial Laplacian and its spectrum}\label{subsec:radial-spectrum}

Consider the unitary group $\U(N)$ with the metric induced by the inner product $\langle X,Y\rangle=N\Tr(XY^*)$ on its Lie algebra $\mathfrak u(N)$, and denote by $\Delta_{\U(N)}$ the associated Laplace--Beltrami operator. Its radial part is the differential operator on the maximal torus
\[
\mathbb T^N=\{(e^{i\theta_1},\ldots,e^{i\theta_N}),\ (\theta_1,\ldots,\theta_N)\in[0,2\pi]^N\}
\]
defined on eigenangles by
\[
\mathscr L_N=\frac1N\sum_{i=1}^N\frac{\partial^2}{\partial\theta_i^2}+\frac1N\sum_{i<j}\cot\left(\frac{\theta_i-\theta_j}{2}\right)\left(\frac{\partial}{\partial \theta_i}-\frac{\partial}{\partial \theta_j}\right),
\]
see, e.g., \cite[Prop.~12.5.1]{Far08} (the definition there changes by a scaling factor due to the metric normalization). The eigenvalues of either $\Delta_{\U(N)}$ or $\mathscr L_N$ are indexed by the set $\widehat{\U(N)}=\{\lambda=(\lambda_1,\ldots,\lambda_N)\in\Z^N:\ \lambda_1\geq\ldots\geq\lambda_N\}$ of highest weights, and given by $-c_{2,N}(\lambda)$ for $\lambda\in\widehat{\U(N)}$, where
\[
c_{2,N}(\lambda)=\frac{1}{N}\sum_{i=1}^N\lambda_i(\lambda_i+N+1-2i)
\]
is the quadratic Casimir eigenvalue associated with $\lambda$. The Hilbert space of square-integrable central functions $L^2(\U(N))^{\mathrm{Ad}}$ can be realized as a space of functions on $\mathbb T^N$, thanks to the Weyl integration formula: one has
\[
L^2(\U(N))^{\mathrm{Ad}}\simeq L^2_{\mathrm{sym}}(\mathbb T^N,\vert\Delta(e^{i\theta})\vert^2\mathrm d\theta),
\]
where $\Delta(z)=\prod_{i<j}(z_i-z_j)$ is the classical Vandermonde determinant. The \emph{central heat trace} is the trace of the heat operator, formally defined by $e^{\frac{t}{2}\Delta_{\U(N)}}$, on this Hilbert space:
\begin{equation}\label{eq:deformed-radial-Laplacian}
\Tr_{L^2(\U(N))^{\mathrm{Ad}}}(e^{\frac{t}{2}\Delta_{\U(N)}})=\Tr_{L^2_{\mathrm{sym}}(\mathbb T^N)}(e^{\frac{t}{2}\mathscr L_N})=\sum_{\lambda\in\widehat{\U(N)}}e^{-\frac{t}{2}c_{2,N}(\lambda)}.
\end{equation}

We now deform this spectral picture, with one caveat: the deformation is applied directly to the radial part of the Laplacian.  As emphasized by Remling--Rosler \cite{RemlingRosler11} in a related framework, deformations of radial Laplacians need not arise from geometric deformations of the ambient space. For $\alpha>0$, define
\begin{equation}\label{eq:weighted-laplacian-expanded}
\mathscr L_{\alpha,N}=\frac{\alpha}{\bigl(N+\alpha-1\bigr)}\sum_{i=1}^N\frac{\partial^2}{\partial\theta_i^2}+\frac{1}{\bigl(N+\alpha-1\bigr)}\sum_{1\leq i<j\leq N}\cot\left(\frac{\theta_i-\theta_j}{2}\right)\left(\frac{\partial}{\partial \theta_i}-\frac{\partial}{\partial \theta_j}\right).
\end{equation}
With this sign convention, $\mathscr L_{\alpha,N}$ is the heat generator and $\mathscr L_{1,N}=\mathscr L_N$.

It is useful to separate the type-$A_{N-1}$ part from the determinant direction.  Let
\[
\partial_{\det}:=\sum_{i=1}^N\frac{\partial}{\partial\theta_i},
\]
and let $\Delta_{A_{N-1}}$ denote the Euclidean Laplacian tangent to the hyperplane $\sum_i\theta_i=0$.  Since
\[
\sum_{i=1}^N\frac{\partial^2}{\partial\theta_i^2}=\Delta_{A_{N-1}}+\frac1N\partial_{\det}^2
\]
and the interaction term in \eqref{eq:weighted-laplacian-expanded} involves only differences of angles, we have
\begin{equation}\label{eq:HO-center-decomposition}
(N+\alpha-1)\mathscr L_{\alpha,N}=\frac{\alpha}{N}\partial_{\det}^2+\left[\alpha\Delta_{A_{N-1}}+\sum_{1\leq i<j\leq N}\cot\left(\frac{\theta_i-\theta_j}{2}\right)\left(\frac{\partial}{\partial\theta_i}-\frac{\partial}{\partial\theta_j}\right)\right].
\end{equation}
The operator in brackets is $\alpha$ times the compact type-$A_{N-1}$ Heckman--Opdam radial operator with root multiplicity $2/\alpha$ (with our cotangent normalization), on the determinant-one torus.  At $\alpha=1$, this is the multiplicity-$2$ operator, and multiplying it by $1/N$ gives the $\mathrm{SU}(N)$ radial Laplacian for the metric inherited from $N\Tr(XY^*)$; the first term in \eqref{eq:HO-center-decomposition} supplies the determinant direction present for $\mathrm U(N)$.

For a partition $\rho$ of length at most $N$, the standard finite-variable Jack eigenvalue identity reads
\begin{equation}\label{eq:finite-Jack-eigenvalue}
-\bigl(N+\alpha-1\bigr)\mathscr L_{\alpha,N}J_\rho^{(\alpha)}=\left(\alpha\sum_{i=1}^N\rho_i^2+\sum_{i=1}^N(N+1-2i)\rho_i\right)J_\rho^{(\alpha)};
\end{equation}
see \cite[Chapter~VI]{Macdonald95}.  To extend this to all signatures, write
$\lambda=m(1,\ldots,1)+\bar\lambda$ with $m=\lambda_N$ and $\bar\lambda$ a partition of length at most $N-1$, and set
\[
J_{\lambda,N}^{(\alpha)}(x)=(x_1\cdots x_N)^mJ_{\bar\lambda}^{(\alpha)}(x).
\]
Since
$\sum_i(N+1-2i)=0$, \eqref{eq:finite-Jack-eigenvalue} then gives
\[
\mathscr L_{\alpha,N}J_{\lambda,N}^{(\alpha)}=-\kappa_{\alpha,N}(\lambda)J_{\lambda,N}^{(\alpha)},
\]
where
\begin{equation}\label{eq:kappa-def}
\kappa_{\alpha,N}(\lambda)=\frac{1}{N+\alpha-1}\left(\alpha\sum_{i=1}^N\lambda_i^2+\sum_{i=1}^N(N+1-2i)\lambda_i\right).
\end{equation}
In particular, $\kappa_{1,N}(\lambda)=c_{2,N}(\lambda)$ is the quadratic Casimir eigenvalue associated with $\lambda$.

For completeness, put
\[
w_{\alpha,N}(\theta)=|\Delta(e^{\mathrm i\theta})|^{2/\alpha},\qquad \mathscr P_N=\C[z_1^{\pm1},\ldots,z_N^{\pm1}]^{\mathfrak S_N}.
\]
On the dense subspace $\mathscr P_N\subset L^2_{\mathrm{sym}}(\mathbb T^N,w_{\alpha,N}\dd\theta)$, the differential expression has divergence form
\[
\mathscr L_{\alpha,N}=\frac{\alpha}{N+\alpha-1}w_{\alpha,N}^{-1}\sum_i\frac{\partial}{\partial \theta_i}\left(w_{\alpha,N}\frac{\partial}{\partial \theta_i}\right),
\]
and is symmetric and nonpositive.  By finite-variable Jack theory and determinant shifts, the Jack Laurent polynomials form a complete orthogonal eigenbasis; see \cite[Chapter~VI]{Macdonald95} and the compact realization in \cite{RemlingRosler11}.  Hence the operator on $\mathscr P_N$ is essentially self-adjoint, and we use its self-adjoint closure throughout.  Moreover,
\[
\sum_i(N+1-2i)\lambda_i=\sum_{k=1}^{N-1}k(N-k)(\lambda_k-\lambda_{k+1})\geq0,
\]
so $\kappa_{\alpha,N}(\lambda)\geq\alpha\sum_i\lambda_i^2/(N+\alpha-1)$.  For fixed $N$, comparison with a Gaussian sum over $\Z^N$ shows that the heat operator is trace class for every $t>0$.  We may therefore define the \emph{Jack heat trace} by
\begin{equation}\label{eq:Z-def}
Z_{\alpha,N}(t):=\Tr_{L^2_{\mathrm{sym}}(\mathbb T^N,\vert\Delta(e^{i\theta})\vert^{2/\alpha}\dd\theta)}\left(e^{\frac{t}{2}\mathscr L_{\alpha,N}}\right)=\sum_{\lambda\in\widehat{\U(N)}}e^{-\frac{t}{2}\kappa_{\alpha,N}(\lambda)}.
\end{equation}

\subsection{Stable representation}\label{subsec:stable-representation}

We use a bijective representation of highest weights described in \cite{Lem22}, which reveals a stable structure of the Casimir eigenvalues $c_{2,N}$.  The same coordinates remain compatible with the Jack deformation and lead to an exact stable decomposition of $\kappa_{\alpha,N}$.  We first recall a few standard facts about integer partitions and Jack polynomials.

Let $\Pcal$ be the set of integer partitions, equivalently Young diagrams.  Given a partition $\lambda\in\Pcal$, define its size $|\lambda|$, length $\ellpart(\lambda)$ and transpose $\lambda'$.  We use the standard statistics
\[
n(\lambda)=\sum_{i\geq1}(i-1)\lambda_i,\qquad n(\lambda')=\sum_{i\geq1}\binom{\lambda_i}{2}.
\]
If $\mu=(1^{m_1}2^{m_2}\cdots)$, we write $m_j(\mu):=m_j$ for the multiplicity of the part $j$, use the standard factor
\begin{equation}\label{eq:z-mu}
z_\mu:=\prod_{j\geq1}j^{m_j(\mu)}m_j(\mu)!,
\end{equation}
and write $p(d)$ for the number of partitions of $d$.
The ordinary total content of a partition is the sum of contents of its boxes, when the partition is represented as a Young diagram:
\begin{equation}\label{eq:ordinary-content}
K(\lambda)=\sum_{(i,j)\in\lambda}(j-i)=n(\lambda')-n(\lambda)=\frac12\sum_i\lambda_i(\lambda_i+1-2i).
\end{equation}

The total content $K(\lambda)$ is the eigenvalue of the usual cut-and-join operator and, in the $\U(N)$ setting of \cite{LemMai25,LM2}, the stable correction to the quadratic Casimir.  Its Jack deformation is obtained by replacing the ordinary box content $j-i$ with the anisotropic content below.  For $\alpha>0$ and a box $\square=(i,j)\in\lambda$, define $c_\alpha(\square)=\alpha(j-1)-(i-1)$ and set
\begin{equation}\label{eq:Calpha}
\cA_\alpha(\lambda)=\sum_{\square\in\lambda}c_\alpha(\square)=\alpha n(\lambda')-n(\lambda).
\end{equation}

We shall use only the following elementary bounds on the Jack content.  For every $\alpha>0$ and $\lambda\in\Pcal$,
\begin{equation}\label{eq:content-crude}
|\cA_\alpha(\lambda)|\leq \frac{\alpha+1}{2}|\lambda|(|\lambda|-1)\leq \frac{\alpha+1}{2}|\lambda|^2.
\end{equation}
Indeed, both $n(\lambda)$ and $n(\lambda')$ are bounded by $\binom{|\lambda|}{2}$.  We shall also use that, if $\ellpart(\lambda)\leq L$,
\begin{equation}\label{eq:K-sharp-lower}
2K(\lambda)\geq \frac{|\lambda|^2}{L}-L|\lambda|.
\end{equation}
This follows immediately from~\eqref{eq:ordinary-content}, Cauchy--Schwarz, and the rearrangement inequality.  We record the estimates explicitly because they are used below, but no special property of Jack theory is involved.

The Jack content is the eigenvalue of the stable Jack cut-and-join operator.  Let $\Lambda=\Q(\alpha)[p_1,p_2,\ldots]$ be the algebra of symmetric functions in power-sum coordinates, and for a partition $\mu$ set $p_\mu:=\prod_i p_{\mu_i}$.  On each homogeneous component $\Lambda_d$, define
\begin{equation}\label{eq:Dalpha}
\Dcal_\alpha=\frac12\left(\alpha\sum_{i,j\geq1}ijp_{i+j}\frac{\partial^2}{\partial p_i\partial p_j}+\sum_{i,j\geq1}(i+j)p_ip_j\frac{\partial}{\partial p_{i+j}}+(\alpha-1)\sum_{i\geq1}i(i-1)p_i\frac{\partial}{\partial p_i}\right).
\end{equation}
The first two terms are the usual join and cut operations, while the last is the Jack deformation.  Note that $\Dcal_\alpha$ preserves degree.  Standard Jack theory gives, for any standard normalization of $J_\lambda^{(\alpha)}$ \cite{Macdonald95,Stanley89},
\begin{equation}\label{eq:Dalpha-eigen}
\Dcal_\alpha J_\lambda^{(\alpha)}=\cA_\alpha(\lambda)J_\lambda^{(\alpha)}.
\end{equation}

Let $\operatorname{ev}_N:\Lambda\to\Q(\alpha)[z_1,\ldots,z_N]^{\mathfrak S_N}$ be finite-variable specialization.  The finite-rank and stable operators satisfy
\begin{equation}\label{eq:L-D}
\mathscr L_{\alpha,N}\operatorname{ev}_N(f) =-d\operatorname{ev}_N(f) -\frac{2}{N+\alpha-1}\operatorname{ev}_N(\Dcal_\alpha f), \qquad f\in\Lambda_d.
\end{equation}
Indeed, for a partition of length at most $N$,
\[
\alpha\sum_i\lambda_i^2+\sum_i(N+1-2i)\lambda_i =(N+\alpha-1)|\lambda|+2\cA_\alpha(\lambda).
\]
The Jack eigenvalue identities prove \eqref{eq:L-D} on a basis; Jack functions of length greater than $N$ specialize to zero on both sides.  Thus the same cut-and-join operator that defines the stable content correction also governs the finite-rank polynomial spectrum.
We now pass from partitions to highest weights using the stable coordinates of \cite{LemMai25,LM2}: fix integers $A_N,B_N\geq0$ such that $A_N+B_N=N-1$ and $A_N,B_N\sim N/2$. For definiteness, throughout the analytic part we take $A_N=\left\lfloor\frac{N-1}{2}\right\rfloor$ and $B_N=N-1-A_N$.  Let $\Pcal_{\leq A}=\{\mu\in\Pcal:\ellpart(\mu)\leq A\}$ be the set of integer partitions of maximum length $A$. If $\mu\in\Pcal_{\leq A_N}$ and $\nu\in\Pcal_{\leq B_N}$, extend both by zero parts and define
\begin{equation}\label{eq:lambda-mun}
\lambda_N(\mu,\nu,m) =\bigl( m+\mu_1,\ldots,m+\mu_{A_N},m, m-\nu_{B_N},\ldots,m-\nu_1 \bigr).
\end{equation}

The map $\lambda_N$ is a bijection from
$\Pcal_{\leq A_N}\times\Pcal_{\leq B_N}\times\Z$ onto $\widehat{\U(N)}$: for a signature $\lambda$, one recovers $m=\lambda_{A_N+1}$ and then
\[
\mu_i=\lambda_i-m,\qquad \nu_j=m-\lambda_{N+1-j}.
\]
Moreover,
\[
\lambda_N(\mu,\nu,m)=\lambda_N(\mu,\nu,0)+m(1,\ldots,1),
\]
so $m$ is the exponent of the determinant twist $(\det)^m$.  We shall refer to $m$ as the \emph{determinant charge}.

\begin{figure}[h!]
\centering
\includegraphics[width=0.3\textwidth]{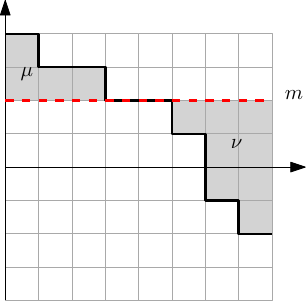}
\caption{Stable decomposition of the signature $(4,3,3,2,2,1,-1,-2)$.  The partitions $\mu$ and $\nu$ encode the positive and negative excitations around the determinant level $m$.}
\label{fig:stable-signature}
\end{figure}

For any $\alpha>0$, any $\mu,\nu\in\Pcal$ and any $m\in\Z$, set
\begin{equation}\label{eq:Ealpha}
E_\alpha(\mu,\nu,m)=|\mu|+|\nu|+\alpha m^2,\qquad F_\alpha(\mu,\nu,m)=\cA_\alpha(\mu)+\cA_\alpha(\nu)+\alpha m(|\mu|-|\nu|)-\frac{\alpha(\alpha-1)}2m^2.
\end{equation}
The deformed Casimir admits the following decomposition.

\begin{proposition}\label{prop:stable-decomposition}
Let $\alpha>0$, $N\geq1$, $\mu\in\Pcal_{\leq A_N}$, $\nu\in\Pcal_{\leq B_N}$, and $m\in\Z$. Then
\begin{equation}\label{eq:stable-exact}
\kappa_{\alpha,N}(\lambda_N(\mu,\nu,m)) =E_\alpha(\mu,\nu,m) +\frac{2}{\bigl(N+\alpha-1\bigr)}F_\alpha(\mu,\nu,m).
\end{equation}
\end{proposition}

\begin{proof}
Set $D=N+\alpha-1$.  Substituting \eqref{eq:lambda-mun} into \eqref{eq:kappa-def}, the two partition sectors satisfy
\[
\alpha\sum_i\mu_i^2+\sum_i(N+1-2i)\mu_i =D|\mu|+2\cA_\alpha(\mu),
\]
and similarly for $\nu$; the reversal in the lower block cancels its minus sign.  The determinant translation contributes
\[
\alpha Nm^2+2\alpha m(|\mu|-|\nu|) =\alpha Dm^2+2\alpha m(|\mu|-|\nu|)-\alpha(\alpha-1)m^2,
\]
while its linear contribution vanishes.  Dividing by $D$ gives \eqref{eq:stable-exact}.
\end{proof}

The decomposition above isolates clearly the dependence in $N$; however, it does not give coercivity because $F_\alpha$ has no fixed sign.  For the Schur case we use the domination estimate of \cite[Lemma~4.1]{LemMai25}: with the same balanced stable coordinates as above,
\begin{equation}\label{eq:Schur-coercivity}
\kappa_{1,N}(\lambda_N(\mu,\nu,m)) \geq \frac12\bigl(|\mu|+|\nu|\bigr) +\left(m+\frac{|\mu|-|\nu|}{N}\right)^2.
\end{equation}
The new point is that this coercivity is stable under the Jack deformation for every positive value of the parameter.  For later use, set $\widetilde\alpha:=\min\{\alpha,\alpha^{-1}\}$.

\begin{lemma}\label{lem:Jack-coercivity}
For every $\alpha>0$, every $N\geq1$, and every dominant signature $\lambda\in\widehat{\U(N)}$,
\begin{equation}
\kappa_{\alpha,N}(\lambda)\geq \widetilde\alpha\kappa_{1,N}(\lambda).
\end{equation}

\end{lemma}

\begin{proof}
For any highest weight $\lambda$, set $L_N(\lambda):=\sum_{i=1}^N(N+1-2i)\lambda_i$.  Dominance gives the useful identity and lower bound
\begin{equation}\label{eq:weyl-linear-positive}
L_N(\lambda)=\sum_{k=1}^{N-1}k(N-k)(\lambda_k-\lambda_{k+1})\geq 0.
\end{equation}
We now bound $\kappa_{\alpha,N}$ using two estimates depending on the value of $\alpha$.

If $0<\alpha\leq1$, then \eqref{eq:kappa-def} gives
\[
\kappa_{\alpha,N}(\lambda) =\frac{\alpha N}{N+\alpha-1}\kappa_{1,N}(\lambda) +\frac{1-\alpha}{N+\alpha-1}L_N(\lambda) \geq \alpha\kappa_{1,N}(\lambda),
\]
because $N+\alpha-1\leq N$.  If $\alpha\geq1$, then
\[
\kappa_{\alpha,N}(\lambda) =\frac{N}{N+\alpha-1}\kappa_{1,N}(\lambda) +\frac{\alpha-1}{N+\alpha-1}\sum_{i=1}^N\lambda_i^2 \geq \alpha^{-1}\kappa_{1,N}(\lambda),
\]
because $N/(N+\alpha-1)\geq1/\alpha$.  The comparison follows.
\end{proof}

Combining the domination estimate~\eqref{eq:Schur-coercivity} with the previous lemma gives the lower bound used below:
\begin{equation}\label{eq:Jack-lower}
\kappa_{\alpha,N}(\lambda_N(\mu,\nu,m)) \geq \widetilde\alpha\left[ \frac12(|\mu|+|\nu|) +\left(m+\frac{|\mu|-|\nu|}{N}\right)^2 \right].
\end{equation}

\subsection{Probabilistic representation}\label{subsec:probabilistic-representation}

Using the stable representation of the Jack-deformed radial Laplacian, we now recast the deformed heat trace in probabilistic terms, following the construction of \cite{LemMai25} for the standard Casimir operator on $\U(N)$.  Set
\[
q_t=e^{-t/2},\qquad \phi(q)=\prod_{k\geq1}(1-q^k),\qquad \Theta_\alpha(t)=\sum_{m\in\Z}e^{-t\alpha m^2/2}.
\]
For $q\in(0,1)$, the $q$-uniform distribution on partitions is
\begin{equation}\label{eq:q-uniform}
\Pbb_q(\lambda)=\phi(q)q^{|\lambda|},\qquad \lambda\in\Pcal.
\end{equation}
Let $\mu,\nu$ be independent partitions with law $\Pbb_{q_t}$, and let $M$ be an independent integer-valued discrete Gaussian random variable:
\begin{equation}\label{eq:discrete-Gaussian}
\Pbb_{\alpha,t}(M=m)=\frac{e^{-t\alpha m^2/2}}{\Theta_\alpha(t)}.
\end{equation}
Let $\Omega_N:=\{\ellpart(\mu)\leq A_N,\ \ellpart(\nu)\leq B_N\}$ be the cutoff event on the lengths of the two random partitions.  We write $\E$ for expectation with respect to the joint product law of $(\mu,\nu,M)$, and $\E_q$ for expectation with respect to a single $q$-uniform partition.  The probabilistic representation is as follows.

\begin{proposition}\label{prop:exact-prob}
For every $N\geq1$, $\alpha>0$, and $t>0$,
\begin{equation}\label{eq:exact-prob}
Z_{\alpha,N}(t) =\frac{\Theta_\alpha(t)}{\phi(q_t)^2} \E\left[ \exp\left(-\frac{t}{\bigl(N+\alpha-1\bigr)}F_\alpha(\mu,\nu,M)\right) \one_{\Omega_N} \right].
\end{equation}
\end{proposition}

\begin{proof}
The bijection $\lambda_N$ and Proposition~\ref{prop:stable-decomposition} give
\[
Z_{\alpha,N}(t) =\sum_{m,\mu,\nu} q_t^{|\mu|+|\nu|}e^{-t\alpha m^2/2} \exp\left(-\frac{t}{N+\alpha-1}F_\alpha(\mu,\nu,m)\right) \one_{\Omega_N}.
\]
Using $q_t^{|\mu|}=\phi(q_t)^{-1}\Pbb_{q_t}(\mu)$ and
$e^{-t\alpha m^2/2}=\Theta_\alpha(t)\Pbb_{\alpha,t}(M=m)$ yields \eqref{eq:exact-prob} by independence.
\end{proof}

We shall use two elementary summability bounds.  For every $q\in(0,1)$ and $s\geq0$ there exist $C,c>0$ such that
\begin{equation}\label{eq:q-tail}
\sum_{\lambda\in\Pcal:|\lambda|>L}|\lambda|^s q^{|\lambda|} \leq Ce^{-cL},\qquad L\geq1,
\end{equation}
with constants locally uniform in $q$.  Indeed, on a compact $q$-interval choose $q<\rho<1$ uniformly; then $d^s q^d\leq C\rho^d$, and Euler's identity
$\sum_{\lambda}\rho^{|\lambda|}=\phi(\rho)^{-1}<\infty$ gives the tail bound.  Likewise, for every compact $K\subset(0,\infty)^2$ and every $s\geq0$,
\begin{equation}\label{eq:gaussian-moments}
\sup_{(t,\alpha)\in K} \sum_{m\in\Z}(1+|m|)^s e^{-t\alpha m^2/4}<\infty,
\end{equation}
because $t\alpha$ is bounded below by a positive constant on $K$.

\begin{lemma}\label{lem:trace-tail}
Fix a compact set $K\subset(0,\infty)^2$ of $(t,\alpha)$-parameters, and let $0<\gamma<1$.  Then there exist $C,c>0$ such that, uniformly on $K$,
\begin{align}\label{eq:trace-tail}
&\sum_{\substack{\mu\in\Pcal_{\leq A_N},\nu\in\Pcal_{\leq B_N},m\in\Z\\
|\mu|+|\nu|>N^\gamma}}
e^{-\frac t2\kappa_{\alpha,N}(\lambda_N(\mu,\nu,m))}
\leq Ce^{-cN^\gamma}.
\end{align}
\end{lemma}

\begin{proof}
By the bound \eqref{eq:Jack-lower}, with $S=|\mu|+|\nu|$ and $d=|\mu|-|\nu|$,
\[
e^{-\frac t2\kappa_{\alpha,N}} \leq e^{-t\widetilde\alpha S/4} e^{-t\widetilde\alpha(m+d/N)^2/2}.
\]
Because $K$ is compact in $(0,\infty)^2$, there are constants $t_0>0$ and $\widetilde\alpha_*>0$ such that $t\geq t_0$ and $\widetilde\alpha\geq\widetilde\alpha_*$ throughout $K$.  Hence $t\widetilde\alpha/2\geq t_0\widetilde\alpha_*/2=:c_1>0$.  The theta series $\sum_{m\in\Z}e^{-c_1(m+x)^2}$ is $1$-periodic and continuous in $x$, hence its supremum over $x\in\R$ is finite.  Summing first over $m$ therefore leaves, up to a uniform constant,
\[
\sum_{\substack{\mu,\nu\\ |\mu|+|\nu|>N^\gamma}}
e^{-c_0(|\mu|+|\nu|)}
\]
with $c_0=t_0\widetilde\alpha_*/4>0$.  Splitting according to whether $|\mu|>N^\gamma/2$ or $|\nu|>N^\gamma/2$ and applying \eqref{eq:q-tail} to each factor gives $Ce^{-cN^\gamma}$ uniformly on $K$.
\end{proof}

\subsection{Asymptotic expansion}\label{subsec:heat-trace-asymptotics}

We can now turn to the main analytic theorem.  The preceding representation has reduced the problem to a perturbation of a fixed product probability measure.  Since the interaction enters as $\exp(-tF_\alpha/\bigl(N+\alpha-1\bigr))$, the formal candidate expansion is its Taylor series.  The estimates previously established are precisely what is needed to justify this Taylor expansion uniformly and to show that the discarded length constraints remain negligible at any order.

\begin{theorem}\label{thm:heat-expansion}
Let $p\in\Z_{\geq0}$.  For every $t>0$ and $\alpha>0$,
\begin{equation}\label{eq:heat-expansion}
Z_{\alpha,N}(t) =\frac{\Theta_\alpha(t)}{\phi(q_t)^2} \sum_{r=0}^{p} \frac{(-t)^r}{r!\bigl(N+\alpha-1\bigr)^r} \E\left[F_\alpha(\mu,\nu,M)^r\right] +O_{t,\alpha,p}(\bigl(N+\alpha-1\bigr)^{-p-1}).
\end{equation}
The estimate is locally uniform in $(t,\alpha)\in(0,\infty)^2$.
\end{theorem}

\begin{proof}
Fix a compact parameter set $K=[t_0,t_1]\times[\alpha_0,\alpha_1]\subset(0,\infty)^2$.  All constants below are uniform on $K$.  Set $D=N+\alpha-1$ and choose $\gamma=1/4$.  Let
\[
T_N=\{|\mu|+|\nu|\leq N^\gamma\}.
\]
For all sufficiently large $N$, $T_N\subset\Omega_N$, because
\[
\ellpart(\mu)\leq|\mu|\leq N^\gamma<A_N, \qquad \ellpart(\nu)\leq|\nu|\leq N^\gamma<B_N.
\]
By Lemma~\ref{lem:trace-tail}, the contribution of $\Omega_N\setminus T_N$ to the exact trace formula is $O(e^{-cN^\gamma})$ and therefore $O(N^{-m})$ for every $m$.

On $T_N$, the crude estimate \eqref{eq:content-crude} and \eqref{eq:Ealpha} give
\begin{equation}\label{eq:F-crude}
|F_\alpha(\mu,\nu,m)| \leq C_{\alpha_1}\left(S^2+|m|S+m^2+1\right), \qquad S=|\mu|+|\nu|.
\end{equation}
Since $S\leq N^{1/4}$ and $D\asymp N$ uniformly on $K$, the terms $S^2/D$ and $1/D$ are bounded.  Moreover, for $N$ large the coefficient of $m^2/D$ in \eqref{eq:F-crude} is at most $\alpha/8$, uniformly on $K$, while Young's inequality gives
\[
\frac{|m|S}{D}\leq \varepsilon m^2+C_\varepsilon\frac{S^2}{D^2}
\]
with $\varepsilon>0$ chosen uniformly small.  Hence, for all sufficiently large $N$,
\begin{equation}\label{eq:F-exp-control}
\frac{t}{D}|F_\alpha(\mu,\nu,m)| \leq \frac{t\alpha}{4}m^2+C_K.
\end{equation}

For $x\in\R$, Taylor's formula gives
\begin{equation}\label{eq:Taylor-bound}
\left|e^{-x}-\sum_{r=0}^{p}\frac{(-x)^r}{r!}\right| \leq \frac{|x|^{p+1}}{(p+1)!}e^{|x|}.
\end{equation}
Apply this with $x=tF_\alpha/D$ inside the exact probabilistic representation.  By \eqref{eq:F-exp-control}, the Gaussian factor from \eqref{eq:discrete-Gaussian} times $e^{|x|}$ is bounded by
\[
C_K e^{-t\alpha m^2/4}.
\]
The factor $|F_\alpha|^{p+1}$ is bounded by a fixed polynomial in $S$ and $|m|$ by \eqref{eq:F-crude}.  The summability bounds \eqref{eq:q-tail} and \eqref{eq:gaussian-moments} therefore imply
\[
\E\left[ |F_\alpha|^{p+1}e^{t|F_\alpha|/D}\one_{T_N} \right]\leq C_{K,p}.
\]
The Taylor remainder is thus $O(D^{-p-1})$.

We have proved
\begin{align*}
Z_{\alpha,N}(t)
&=\frac{\Theta_\alpha(t)}{\phi(q_t)^2}
\sum_{r=0}^{p}\frac{(-t)^r}{r!D^r}
\E\left[F_\alpha^r\one_{T_N}\right]
+O(D^{-p-1}).
\end{align*}
Finally, by \eqref{eq:F-crude} and the exponential tails of the product $q_t$-uniform law together with the discrete Gaussian law,
\[
\E\left[|F_\alpha|^r\one_{T_N^c}\right] =O(e^{-cN^\gamma})
\]
for each fixed $r$.  We may therefore remove $\one_{T_N}$ at an error smaller than any power of $N$.  This proves \eqref{eq:heat-expansion}, uniformly on $K$.
\end{proof}

\begin{remark}
The leading term in Theorem~\ref{thm:heat-expansion} is
\[
\lim_{N\to\infty}Z_{\alpha,N}(t)=\frac{\Theta_{\alpha}(t)}{\phi(q_t)^2}.
\]
For the first correction, set
\begin{equation}\label{eq:S-q}
S(q):=\sum_{k\geq1}\binom{k}{2}\frac{q^k}{1-q^k}.
\end{equation}
The independent geometric laws of the part multiplicities under the $q$-uniform measure, together with transpose invariance, give
\begin{equation}\label{eq:mean-Jack-content}
\E_q[\cA_\alpha(\mu)]=(\alpha-1)S(q).
\end{equation}
Moreover $\E[M^2]=-(2/\alpha)\partial_t\log\Theta_\alpha(t)$.  Using the symmetry of $M$ in the definition of $F_\alpha$ therefore yields
\begin{equation}\label{eq:EF}
\E[F_\alpha]=(\alpha-1)\left(2S(q_t)+\partial_t\log\Theta_\alpha(t)\right).
\end{equation}
Hence
\begin{equation}\label{eq:first-correction-explicit}
Z_{\alpha,N}(t) =\frac{\Theta_\alpha(t)}{\phi(q_t)^2} \left[ 1-\frac{t(\alpha-1)}{N+\alpha-1} \left(2S(q_t)+\partial_t\log\Theta_\alpha(t)\right) +O_{t,\alpha}(N^{-2}) \right].
\end{equation}
At $\alpha=1$ the order-$N^{-1}$ term vanishes, in agreement with \cite{LM2}.
\end{remark}

\section{Elliptic \texorpdfstring{$b$}{b}-Hurwitz theory}\label{sec:elliptic-b-hurwitz}

In this section we construct the genus-one counterpart of the generalized branched-cover theory of Chapuy--Do\l\k{e}ga \cite{ChapuyDolega22} from several perspectives.  The starting point is the simplest stable spectral quantity attached to the Jack cut-and-join operator: the power sum of its total-content eigenvalues.  We take this quantity as the definition of the elliptic theory and show below that it admits three equivalent realizations: as the genus-one closure of the simple $b$-Hurwitz theory, as a weighted count of real elliptic generalized coverings, and as an operator trace.

\begin{definition}\label{def:Hb}
For $d,r\geq0$, the \emph{genus-one $b$-Hurwitz number} is
\begin{equation}\label{eq:Hb}
\Hb_{1,b}(d,r) :=\sum_{\lambda\vdash d}\cA_{1+b}(\lambda)^r.
\end{equation}
For $d=0$ we use the convention $\Hb_{1,b}(0,0)=1$ and $\Hb_{1,b}(0,r)=0$ for $r>0$.
\end{definition}

We package these numbers into the generating series
\begin{equation}\label{eq:elliptic-master-H}
\mathscr H_b(q;\hbar) :=\sum_{d,r\geq0}\Hb_{1,b}(d,r)q^d\frac{\hbar^r}{r!}.
\end{equation}
Here $d$ is the generalized degree and $r$ is the number of labelled simple branch values in the geometric realizations below.

\subsection{Genus-one closure of the simple $b$-Hurwitz theory}\label{subsec:genus-one-closure}

We first construct the elliptic $b$-Hurwitz theory as the genus-one closure of the simple $b$-Hurwitz theory. Let us recall the required parts of the genus-zero theory from \cite{ChapuyDolega22}: generalized coverings and constellations, their MON weights in the simple-Hurwitz specialization, the disconnected automorphism normalization, and the Jack two-boundary kernel.

Let $\mathbb S^2_+$ be the closed upper hemisphere and let $\varpi_0:\mathbb S^2\to\mathbb S^2_+$ identify the two hemispheres by reflection across the equator.  If $\Sigma$ is a possibly nonorientable compact surface, write $\pi_\Sigma:\widetilde\Sigma\to\Sigma$ for its orientation double cover with deck involution $\jmath$.  A \emph{generalized branched covering} is a continuous map $f:\Sigma\to\mathbb S^2_+$, covering over the interior, admitting an ordinary branched covering $\widetilde f:\widetilde\Sigma\to\mathbb S^2$ such that
\[
f\circ\pi_\Sigma=\varpi_0\circ\widetilde f.
\]
Its generalized degree $d$ is half the degree of $\widetilde f$.  If $\lambda\vdash d$, a generalized branch value has profile $\lambda$ when its lift has ordinary profile
$[2]\lambda=(\lambda_1,\lambda_1,\ldots,\lambda_\ell,\lambda_\ell)$.

A map is a graph cellularly embedded in a compact surface; a corner is a sector between consecutive half-edges.  Equivalently, such an embedded graph determines a ribbon graph.  A $k$-constellation is such a map with vertices colored by $\{0,\ldots,k\}$, with a consecutive-color incidence rule; its size is the number of color-$0$ corners. Examples of maps, ribbon graphs and constellations are displayed in Figure~\ref{fig:map-ribbon}.

\begin{figure}[h!]
    \centering
    \includegraphics[width=0.8\textwidth]{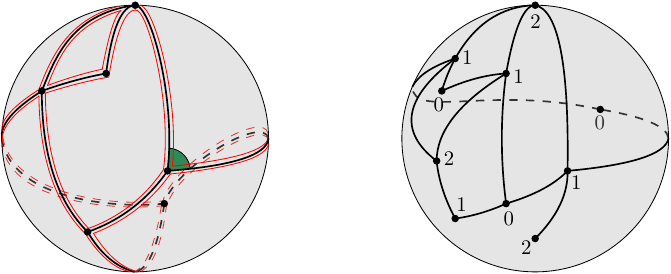}
    \caption{Left: an arbitrary map on the sphere and its thickening into a ribbon graph.  The highlighted sector between two consecutive half-edges illustrates a corner. Right: a 2-constellation of size 7 on the sphere.}
    \label{fig:map-ribbon}
\end{figure}

\begin{proposition}[\cite{ChapuyDolega22}, Proposition~2.3]\label{prop:covering-constellation-correspondence}
Let $f:\Sigma\to\mathbb S_+^2$ be a generalized branched covering with $k+2$ branch values monotonically numbered from $-1$ to $k$ along the equator.  If $P$ is the equatorial path through $0,1,\ldots,k$ in this order, then $f^{-1}(P)$ is a $k$-constellation on $\Sigma$, and this construction is bijective relative to the full profiles.
\end{proposition}

In the simple $b$-Hurwitz specialization, a constellation of type $(\mu,\nu;r)$ has distinguished profiles $\mu,\nu\vdash d$ and $r$ further labelled simple branch values.  A labelled oriented presentation marks the $d$ color-$0$ corners and chooses a local orientation at each, hence gives $2^d d!$ presentations before quotienting by automorphisms.

A \emph{measure of non-orientability} (MON) assigns a polynomial $\rho(M,e)\in\Q[b]$ to the deletion of an edge $e$ in a map $M$.  If deleting $e$ exposes corners $c_1,c_2$, its local rules are: weight $1$ when the corners lie in distinct components; for two reconnections between different faces of one component, the two weights sum to $1+b$; and for two reconnections inside one face, the face-splitting choice has weight $1$ and its twisted alternative weight $b$.  Coherence is the corresponding cyclic compatibility condition of \cite[Definition~3.2]{ChapuyDolega22}; coherent integral MONs exist, with all local values in $\{1,b\}$.  Chapuy--Do\l\k{e}ga's recursive right-path deletion then gives a weight $\widetilde\rho(C,\omega)$ to every labelled oriented connected constellation.  We extend it multiplicatively to disconnected constellations, with each component retaining only the labels of the simple values at which it ramifies.  The resulting generating series is independent of the coherent MON. Throughout this section, a coherent MON $\rho$ shall be fixed, and it will be implicit in notations. When positivity or the specializations $b=0,1$ are invoked, $\rho$ is chosen integral.  The numerical invariants constructed below will be independent of this choice after summation.

For fixed $d,r$ and $\mu,\nu\vdash d$, we use the disconnected coefficient normalized by
\begin{equation}\label{eq:CD-labelled-expansion}
\Hb_{0,b}(\mu,\nu,r) =\frac{1}{2^d d!} \sum_{[\widehat C]} \left(\frac{2}{1+b}\right)^{c(\widehat C)} \widetilde\rho(\widehat C),
\end{equation}
where the sum is over labelled oriented disconnected double constellations of type $(\mu,\nu;r)$.  For an unlabelled constellation $C$, let $\mathrm{Lab}(C)$ be its set of labelled oriented presentations and set
\begin{equation}\label{eq:CD-unlabelled-weight}
w_b(C) :=\left(\frac{2}{1+b}\right)^{c(C)} \frac1{2^d d!}\sum_{\omega\in\mathrm{Lab}(C)}\widetilde\rho(C,\omega), \qquad w_b(\varnothing)=1.
\end{equation}
The label convention makes this weight multiplicative:
\begin{equation}\label{eq:spherical-weight-multiplicative}
w_b(C_1\sqcup C_2)=w_b(C_1)w_b(C_2).
\end{equation}
For later use, set
\[
\overline\rho_b(C):=\left(\frac{1+b}{2}\right)^{c(C)}w_b(C).
\]
For an integral MON, $\overline\rho_b(C)\in\Q_{\geq0}[b]$; it equals $1$ at $b=1$, and at $b=0$ it is $1$ exactly when every source component is orientable.

\begin{lemma}\label{lem:CD-unlabelled-normalization}
For fixed $d,r\geq0$ and $\mu,\nu\vdash d$,
\begin{equation}\label{eq:CD-enumerative-definition}
\Hb_{0,b}(\mu,\nu,r) =\sum_{[C]}\frac{w_b(C)}{|\Aut(C)|},
\end{equation}
where $C$ runs over unlabelled disconnected double constellations of type $(\mu,\nu;r)$.
\end{lemma}

\begin{proof}
Group the presentations in \eqref{eq:CD-labelled-expansion} above a fixed unlabelled $C$.  The action of $\Aut(C)$ on $\mathrm{Lab}(C)$ is free, since fixing all oriented color-$0$ corners fixes the colored cellular map. Hence, the total contribution above $C$ is exactly $w_b(C)/|\Aut(C)|$.  Summing over $[C]$ proves the formula.  Removing the two distinguished caps changes neither this weight nor the automorphism group, so the same normalization applies to the cut presentation below.
\end{proof}

Finally, we recall the two-boundary kernel used in the closure below.  Put $\alpha=1+b$ and use the integral Jack normalization.  This kernel is precisely
\begin{equation}\label{eq:CD-kernel}
\begin{aligned}
\mathsf K_{d,\alpha}(\hbar;\mathbf p,\mathbf y):=\sum_{\lambda\vdash d}\frac{J_\lambda^{(\alpha)}(\mathbf p)J_\lambda^{(\alpha)}(\mathbf y)}{j_\lambda^{(\alpha)}} e^{\hbar\cA_\alpha(\lambda)}=\sum_{\mu,\nu\vdash d}p_\mu(\mathbf p)p_\nu(\mathbf y)\sum_{s\geq0}\frac{\hbar^s}{s!}\Hb_{0,b}(\mu,\nu,s),
\end{aligned}
\end{equation}
where $j_\lambda^{(\alpha)}=\langle J_\lambda^{(\alpha)},J_\lambda^{(\alpha)}\rangle_\alpha$; see \cite[Equation~(63) and Theorem~6.4]{ChapuyDolega22}.  At $\hbar=0$, the Jack Cauchy identity gives
\begin{equation}\label{eq:inverse-Jack-Cauchy-kernel}
\mathsf K_{d,\alpha}(0;\mathbf p,\mathbf y) =\sum_{\mu\vdash d}\frac{p_\mu(\mathbf p)p_\mu(\mathbf y)}{z_\mu\alpha^{\ellpart(\mu)}},
\end{equation}
where $z_\mu$ is defined in~\eqref{eq:z-mu}, hence $\langle p_\mu,p_\nu\rangle_\alpha =\delta_{\mu\nu}z_\mu\alpha^{\ellpart(\mu)}.$

We now identify Definition~\ref{def:Hb} with the genus-one closure of the simple $b$-Hurwitz kernel recalled above; by the genus-one closure we mean precisely sewing the two boundary components of the cylinder, turning the genus-zero target into a genus-one target.

\begin{proposition}\label{prop:cylinder-closure}
For every $d,r\geq0$,
\begin{equation}\label{eq:Hb-cylinder-definition}
\Hb_{1,b}(d,r)=\sum_{\mu\vdash d}z_\mu(1+b)^{\ellpart(\mu)}\Hb_{0,b}(\mu,\mu,r).
\end{equation}
\end{proposition}

\begin{proof}
Put $\alpha=1+b$.  By \eqref{eq:CD-kernel}, the coefficient of $p_\mu(\mathbf p)p_\nu(\mathbf y)$ in the cylinder propagator is the exponential generating series of $\Hb_{0,b}(\mu,\nu,r)$.  Closing the two boundaries contracts equal profiles with the Jack scalar product \eqref{eq:inverse-Jack-Cauchy-kernel}; hence the closed amplitude is
\[
\sum_{\mu\vdash d}z_\mu\alpha^{\ellpart(\mu)} \sum_{r\geq0}\frac{\hbar^r}{r!}\Hb_{0,b}(\mu,\mu,r).
\]
Applying the same contraction to the first line of \eqref{eq:CD-kernel} and using Jack orthogonality gives
\[
\sum_{\lambda\vdash d}e^{\hbar\cA_\alpha(\lambda)}.
\]
Comparing coefficients of $\hbar^r/r!$ with Definition~\ref{def:Hb} proves \eqref{eq:Hb-cylinder-definition}.
\end{proof}

This closure admits an algebraic interpretation: it is the trace of an explicit operator.

\begin{proposition}\label{prop:cut-and-join-trace}
For every $d,r\geq0$,
\begin{equation}\label{eq:Hb-trace}
\Hb_{1,b}(d,r)=\Tr_{\Lambda_d}(\Dcal_{1+b}^{r}).
\end{equation}
Consequently,
\begin{equation}\label{eq:elliptic-master-spectral}
\mathscr H_b(q;\hbar)=\sum_{\lambda\in\Pcal}q^{|\lambda|}e^{\hbar\cA_{1+b}(\lambda)}.
\end{equation}
\end{proposition}

\begin{proof}
By \eqref{eq:Dalpha-eigen}, the Jack basis diagonalizes $\Dcal_{1+b}$ with eigenvalues $\cA_{1+b}(\lambda)$.  Hence
\[
\Tr_{\Lambda_d}(\Dcal_{1+b}^{r}) =\sum_{\lambda\vdash d}\cA_{1+b}(\lambda)^r =\Hb_{1,b}(d,r)
\]
by Definition~\ref{def:Hb}.  Exponentiating and summing over $d$ with weight $q^d$ gives \eqref{eq:elliptic-master-spectral}.
\end{proof}

We record two algebraic consequences of the trace realization. First, the standard Jack involution gives
\begin{equation}\label{eq:elliptic-Jack-duality}
\Hb_{1,\alpha-1}(d,r) =(-\alpha)^r\Hb_{1,\alpha^{-1}-1}(d,r),
\end{equation}
since $\cA_\alpha(\lambda')=-\alpha\cA_{\alpha^{-1}}(\lambda)$.  In particular, $\Hb_{1,0}(d,r)=0$ for odd $r$ by transpose symmetry. Second, the elliptic invariants inherit the positivity phenomenon of the genus-zero theory.

\begin{theorem}\label{thm:positivity}
For all $d,r\geq0$, $\Hb_{1,b}(d,r)\in\Z_{\geq0}[b]$.
\end{theorem}

\begin{proof}
Write $m_i=m_i(\mu)$ for the multiplicities in $p_\mu$.  Combining the ordered pairs $(i,j)$ and $(j,i)$ in \eqref{eq:Dalpha}, a join of distinct parts $i<j$ has coefficient $(1+b)ijm_im_j$, and a join of equal parts has coefficient $(1+b)i^2\binom{m_i}{2}$.  A cut of a part $k=i+j$ into distinct parts has coefficient $k m_k$, while the equal cut $2i\mapsto(i,i)$ has coefficient $i m_{2i}$.  The diagonal coefficient is $b\sum_i\binom{i}{2}m_i$.  All these coefficients lie in $\Z_{\geq0}[b]$.  Thus every entry of the matrix of $\Dcal_{1+b}$ in the basis $(p_\mu)_{\mu\vdash d}$ belongs to this semiring, and the same is true for every entry of each of its powers; hence their traces also lie in $\Z_{\geq0}[b]$.  Proposition~\ref{prop:cut-and-join-trace} proves the theorem, including $r=0$.
\end{proof}

\subsection{Real elliptic generalized coverings}\label{subsec:real-elliptic-covers}

The genus-one analogue uses an oriented torus with an orientation-reversing involution $\iota$ whose fixed locus $E^\iota$ has two components and whose quotient $E/\iota$ is an annulus.  We use the standard real elliptic model
\begin{equation}\label{eq:real-elliptic-target}
E_T:=\C/(\Z+iT\Z),\qquad \iota[z]:=[\overline z],\qquad T>0.
\end{equation}
Its fixed locus is the disjoint union of the two real ovals $R_0=\R/\Z$ and $R_1=(\R+iT/2)/\Z$, and the quotient $\varpi:E_T\longrightarrow Q_T:=E_T/\iota$ is a closed annulus.  Every such involution is topologically equivalent to this model, and the parameter $T$ plays no enumerative role.  This geometry is illustrated in Figure~\ref{fig:real-elliptic-annulus}.

\begin{figure}[h!]
    \centering
    \includegraphics[width=0.8\linewidth]{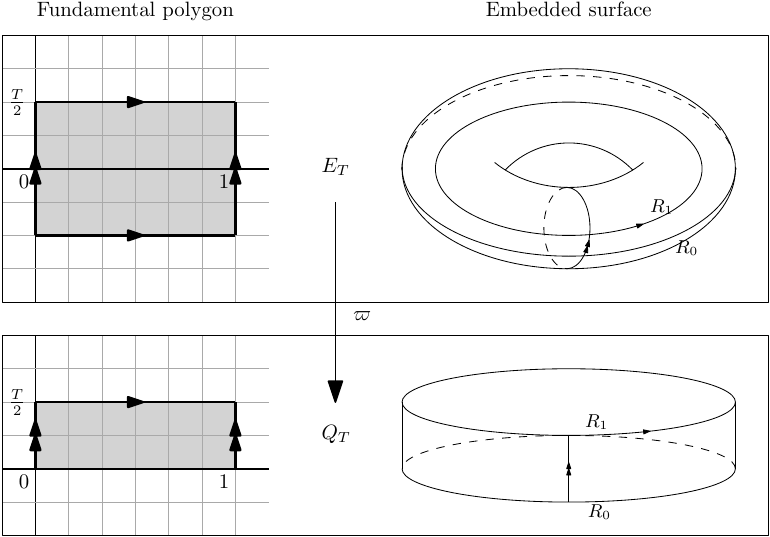}
    \caption{The real elliptic target and its annular quotient.  The upper panels represent the torus $E_T$, while the lower panels represent $Q_T=E_T/\iota$; the fixed ovals $R_0$ and $R_1$ become the two boundary components of the quotient annulus.}
    \label{fig:real-elliptic-annulus}
\end{figure}

\begin{definition}\label{def:real-elliptic-cover}
A degree-$d$ \emph{real elliptic generalized covering} of $(E_T,\iota)$ is a commutative diagram
\begin{equation}\label{eq:real-elliptic-cover-diagram}
\begin{array}{ccc}
\widetilde\Sigma&\xrightarrow{\ \widetilde F\ }&E_T\\
\downarrow\pi_\Sigma&&\downarrow \varpi\\
\Sigma&\xrightarrow{\ F\ }&Q_T
\end{array}
\end{equation}
where $\pi_\Sigma:\widetilde\Sigma\to\Sigma$ is the orientation double cover with fixed-point-free deck involution $\jmath$, the restriction
\[
F:\Sigma\setminus F^{-1}(\partial Q_T)\longrightarrow\operatorname{int}(Q_T)
\]
is an ordinary $2d$-sheeted covering, and $\widetilde F$ is an ordinary branched covering of degree $2d$, all of whose branch values lie in $E_T^\iota$, satisfying
\begin{equation}\label{eq:real-equivariance}
\widetilde F\circ\jmath=\iota\circ\widetilde F.
\end{equation}
A generalized branch value is a point of $\partial Q_T=\varpi(E_T^\iota)$ whose lift to $E_T^\iota$ has doubled ramification profile $[2]\lambda$ for $\widetilde F$; its profile is $\lambda\vdash d$.  It is simple when $\lambda=(2,1^{d-2})$.  For $d<2$ this simple profile does not exist: whenever $r>0$ we therefore declare the corresponding simple-branching set to be empty, while for $r=0$ we retain the unramified degree-$d$ objects (including the empty object for $d=0$).  Isomorphisms are homeomorphisms of the source commuting with the whole diagram.  Disconnected objects are finite disjoint unions, with the same conventions as above.
\end{definition}

As in the genus-zero case, the generalized degree is half the degree on the orientation double cover.  Equivariance pairs distinct points $x,\jmath x$ above every fixed point of $\iota$ with the same local degree, which explains the doubled profile $[2]\lambda$. Globally, the annulus $Q_T$ plays in genus one the role played by the hemisphere in the genus-zero construction: the quotient of the real sphere by its standard orientation-reversing involution is a disk, whose boundary is the real equator, whereas the quotient of the torus is an annulus, whose two boundary components are the two real ovals. The essential difference is that the annulus retains one nontrivial cycle.  This is precisely the extra topology that must be removed before reducing to the spherical theory: cutting $Q_T$ along a spanning arc joining its two boundary components produces a disk, while upstairs the lift of this arc is a nonseparating circle whose removal turns the torus into a cylinder.  The two sides created by this cut are responsible for the equal profiles $(\mu,\mu)$ that appear below.

For the enumerative problem, fix distinct labelled points $\mathbf x=(x_1,\ldots,x_r)\subset \varpi(R_0)\subset\partial Q_T$.  Let $\mathfrak H_{d,r}(E_T,\iota;\mathbf x)$ denote the set of isomorphism classes of possibly disconnected degree-$d$ real elliptic generalized coverings for which the $x_i$ are exactly the generalized branch values and each has simple profile $(2,1^{d-2})$, with the $d<2$ convention above.  For a representative $F$ we write $\Aut(F)$ for its finite automorphism group.  Up to isotopy preserving the boundary components, the precise positions of the branch values on this fixed boundary component are immaterial; relabelling shows that the numerical count below is independent of their prescribed cyclic order.

\begin{remark}
Each generalized simple branch value gives two ordinary simple ramification points upstairs.  Since $\chi(E_T)=0$ and $\widetilde\Sigma$ is the orientation double cover,
\begin{equation}\label{eq:real-RH}
\chi(\widetilde\Sigma)=-2r, \qquad \chi(\Sigma)=-r.
\end{equation}
\end{remark}

To obtain a cylinder presentation of a real elliptic generalized covering, choose a properly embedded spanning arc $a\subset Q_T$ joining the two boundary components and disjoint from the branch values.  We call such a spanning arc an \emph{admissible cut}.  Put
\begin{equation}\label{eq:invariant-seam}
\gamma_a:=\varpi^{-1}(a)\subset E_T.
\end{equation}
Then $\gamma_a$ is an $\iota$-invariant nonseparating circle, and cutting $E_T$ along $\gamma_a$ gives an oriented cylinder whose quotient is the disk obtained by cutting $Q_T$ along $a$.  No connected component $C$ of $\widetilde F^{-1}(\gamma_a)$ can be invariant under $\jmath$.  Indeed, $\widetilde F|_C:C\to\gamma_a$ is a covering and $\iota|_{\gamma_a}$ is orientation reversing; by equivariance, $\jmath|_C$ would then be an orientation-reversing involution of the circle $C$.  Every orientation-reversing homeomorphism of $S^1$ has a fixed point, contradicting the fact that $\jmath$ is the fixed-point-free deck involution of the orientation double cover.  The seam circles occur therefore in $\jmath$-exchanged pairs of equal degree.  More explicitly, write
\[
\widetilde F^{-1}(\gamma_a) =\bigsqcup_{s=1}^{\ell}(C_s\sqcup\jmath C_s), \qquad \deg(C_s\to\gamma_a)=j_s.
\]
Since $\widetilde F$ has degree $2d$, one has $\sum_s j_s=d$, and the seam therefore determines the partition
\[
\mu=(j_1,\ldots,j_{\ell})=(1^{m_1}2^{m_2}\cdots)\vdash d.
\]
After cutting, each $C_s$ produces a degree-$j_s$ boundary circle at each end of the cylinder.  Capping a target boundary circle by a disk extends this boundary cover, uniquely up to isomorphism, by the standard branched disk map $z\mapsto z^{j_s}$.  Hence the center of each target cap has generalized profile $\mu$; on the orientation double cover its ordinary profile is $[2]\mu$, with every part $j_s$ occurring twice.  The two profiles are automatically equal because the two cylinder boundaries are the two copies of the same cut seam.  Thus equivariant capping produces a generalized double covering of the sphere with distinguished profiles $(\mu,\mu)$.  The converse sewing construction and its automorphisms are made precise in Proposition~\ref{prop:cut-cap} below.  Figure~\ref{fig:real-elliptic-cut-cap} summarizes this passage.

\begin{figure}[h!]
\centering
\includegraphics[width=0.95\textwidth]{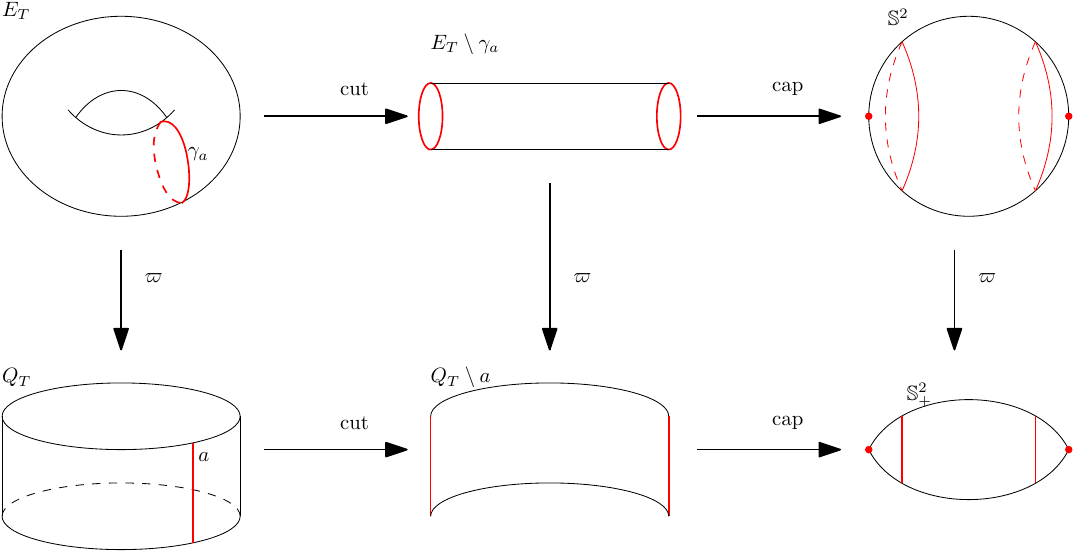}
\caption{Cut--cap reduction of the real elliptic target to the spherical theory.  Cutting the torus along the lift $\gamma_a$ of a spanning arc $a$ in the quotient annulus produces a cylinder with two distinguished boundary circles; capping these boundaries produces a sphere.  The two copies of the cut seam give equal distinguished generalized profiles over the cap centers.}
\label{fig:real-elliptic-cut-cap}
\end{figure}

We next realize this genus-one closure as the weighted count of genuine real elliptic generalized coverings.  Let $C$ have equal distinguished profile $\mu$.  After removing the two caps, each degree-$j$ part $B$ determines a pair $(C_B^0,C_B^1)$ of degree-$j$ boundary circle covers exchanged by $\jmath$.  Denote by $B_j^+(C)$ and $B_j^-(C)$ the sets of degree-$j$ parts at the two distinguished values, respectively; thus
\[
|B_j^+(C)|=|B_j^-(C)|=m_j(\mu).
\]

\begin{definition}\label{def:closure-datum}
For two degree-$j$ boundary orbits $B^+,B^-$ at opposite ends, let $\operatorname{Iso}_{\mathrm{eq}}(B^+,B^-)$ be the set of isomorphisms between their two-circle covers of the seam, commuting with the deck involutions and covering the prescribed identification of the target boundaries.  The finite set $\operatorname{Sew}(C)$ of \emph{equivariant elliptic sewing data} consists of a degree-preserving matching $\psi_j:B_j^+(C)\to B_j^-(C)$ for every $j$, together with an element of $\operatorname{Iso}_{\mathrm{eq}}(B,\psi_j(B))$ for each matched orbit $B$.
\end{definition}

After temporarily ordering each pair of circles, an equivariant isomorphism has a direct or crossed pairing.  For either pairing, the isomorphisms of one degree-$j$ circle cover form a torsor under its cyclic deck group, and the second isomorphism is forced by equivariance.  As a consequence, there are $2j$ choices.  The group $\Aut(C)$ acts by transport of the matching and conjugation of the boundary isomorphisms.

\begin{proposition}\label{prop:cut-cap}
Fix an admissible spanning arc $a\subset Q_T$.  Cut--cap induces a bijection
\begin{equation}\label{eq:cut-cap-classification}
\mathfrak H_{d,r}(E_T,\iota;\mathbf x) \simeq \bigsqcup_{\mu\vdash d} \ \bigsqcup_{\substack{[C]\text{ unlabelled of type }(\mu,\mu;r)}} \operatorname{Sew}(C)/\Aut(C).
\end{equation}
\end{proposition}

\begin{proof}
Cutting the target along $\gamma_a$ and the source along $\widetilde F^{-1}(\gamma_a)$ gives an equivariant cylinder covering.  By the preceding discussion, its seam circles occur in $\jmath$-exchanged pairs of equal degree and determine a partition $\mu\vdash d$.  Capping a degree-$j$ boundary circle uses, up to isomorphism, the unique branched disk cover $z\mapsto z^j$; after equivariant capping and passage to the quotient, Proposition~\ref{prop:covering-constellation-correspondence} therefore gives an unlabelled double constellation $C$ of type $(\mu,\mu;r)$.

Capping forgets exactly how equal-degree boundary orbits at the two ends were identified.  For each matched degree-$j$ pair there are precisely the $2j$ equivariant identifications forming $\operatorname{Iso}_{\mathrm{eq}}(B^+,B^-)$, so the lost data are exactly an element $g\in\operatorname{Sew}(C)$.  Conversely, removing the two distinguished caps and sewing according to $g$ reconstructs the cylinder covering and hence the original real elliptic covering.  Thus isomorphism classes are precisely the $\Aut(C)$-orbits in $\operatorname{Sew}(C)$, proving \eqref{eq:cut-cap-classification}.

\end{proof}

The cut--cap correspondence is also compatible with automorphisms.  If a real elliptic cover $F$ corresponds to the $\Aut(C)$-orbit of $g\in\operatorname{Sew}(C)$, restriction and extension across the standard caps give a canonical identification
\[
\Aut(F)\simeq\operatorname{Stab}_{\Aut(C)}(g)=\Aut(C,g).
\]
Indeed, an automorphism of $C$ descends to the sewn cover precisely when it preserves the sewing datum $g$, while cutting an automorphism of $F$ gives such a stabilizing automorphism of $C$.

\subsection{Sewing and the geometric realization}\label{subsec:geometric-sewing}

We now place the $b$-weight on the sewing presentations supplied by the cut--cap correspondence and prove that their automorphism-weighted sum reproduces the genus-one closure of Section~\ref{subsec:genus-one-closure}.

The spherical MON supplies the weight of the capped object.  To sew entire boundary circles we use the following explicit rule on surfaces with boundary; this rule requires no extension of an edge-deletion MON to a noncellular or partially sewn map.

\begin{lemma}\label{lem:boundary-sewing-topology}
Fix a matching of two degree-$j$ boundary orbits in a partially sewn presentation.  Its $2j$ equivariant boundary isomorphisms have the following effects on the quotient source.
\begin{enumerate}[label=\textup{(\roman*)}]
\item If the boundary circles belong to distinct connected components, every sewing joins those components.  The result is orientable exactly when both components were orientable.
\item If they belong to one orientable component, exactly $j$ sewings keep it orientable and exactly $j$ make it nonorientable.
\item If they belong to one nonorientable component, every sewing leaves it nonorientable.
\end{enumerate}
In the last two cases the number of components is unchanged.  These classifications are intrinsic and are preserved by isomorphisms of the presentation.
\end{lemma}

\begin{proof}
The quotient of a $\jmath$-exchanged pair of seam circles is one boundary circle.  The direct and crossed equivariant pairings induce the two possible orientation types of its identification with the matched circle.  Within either type the $j$ choices differ by a deck rotation, so are isotopic as boundary identifications and have the same topological effect.  On an orientable component, fix either orientation.  Gluing two of its boundary circles gives an orientable surface exactly when the identification reverses their induced boundary orientations; the other orientation type gives a nonorientable surface.  The criterion is unchanged if the chosen orientation is reversed.  For two distinct orientable components one may reverse the orientation of either component, so both types of identification give an orientable union.  A nonorientable component remains nonorientable after boundary sewing, since an orientation of the result would restrict to an orientation of its interior.  Finally, identifying two boundary circles joins their components exactly when those components were distinct.  This proves the assertions.
\end{proof}

Let $m=\ellpart(\mu)$ and fix a total order $\prec$ of the positive-end boundary orbits.  Remove all distinguished caps and sew according to $g\in\operatorname{Sew}(C)$ in this order.  The quotient source at each stage is a compact surface with boundary.  We assign the $k$-th sewing the factor
\begin{equation}\label{eq:boundary-sewing-rule}
s_k(g,\prec)=
\begin{cases}
1,&\text{one orientable component remains orientable},\\
b,&\text{one orientable component becomes nonorientable},\\
(1+b)/2,&\text{the circles lie in one nonorientable component},\\
(1+b)/2,&\text{the circles lie in distinct components}.
\end{cases}
\end{equation}
In the last case the factor agrees with the change of component normalization,
\[
\frac{(2/(1+b))^{c-1}}{(2/(1+b))^c}=\frac{1+b}{2}.
\]
In every case, Lemma~\ref{lem:boundary-sewing-topology} gives
\begin{equation}\label{eq:one-boundary-sum}
\sum_{h\in\operatorname{Iso}_{\mathrm{eq}}(B^+,B^-)}s(h)=j(1+b),
\end{equation}
where the preceding partial sewing is held fixed.

Define
\begin{equation}\label{eq:order-averaged-sewing-weight}
s_b(g):=\frac1{m!}\sum_{\prec}\prod_{k=1}^m s_k(g,\prec),
\end{equation}
with $s_b(g)=1$ for $m=0$, where the sum is over all total orders of the positive-end boundary orbits.  Since \eqref{eq:boundary-sewing-rule} depends only on connectivity and orientability, an automorphism of $C$ permutes these orders and carries every ordered sewing sequence to one with the same product of local factors.  Hence
\[
s_b(\varphi\cdot g)=s_b(g), \qquad \varphi\in\Aut(C).
\]
Thus, for the fixed cut $a$, the quantity
\[
w_{b,a}(F):=w_b(C)s_b(g)
\]
is well defined on the isomorphism class of the real elliptic cover $F$ represented by $(C,g)$.  We define
\begin{equation}\label{eq:real-weighted-cardinality}
\#_{b,a}\mathfrak H_{d,r}(E_T,\iota;\mathbf x) :=\sum_{[F]\in\mathfrak H_{d,r}(E_T,\iota;\mathbf x)} \frac{w_{b,a}(F)}{|\Aut(F)|}.
\end{equation}
\begin{theorem}\label{thm:geometric-gluing}
For every $d,r\geq0$, $b\geq0$, $T>0$, admissible cut $a$ and coherent MON,
\begin{equation}\label{eq:geometric-diagonal-gluing}
\#_{b,a}\mathfrak H_{d,r}(E_T,\iota;\mathbf x) =\Hb_{1,b}(d,r) =\sum_{\mu\vdash d} z_\mu(1+b)^{\ellpart(\mu)} \Hb_{0,b}(\mu,\mu,r).
\end{equation}
\end{theorem}

\begin{proof}
Fix $C$ with common profile $\mu=(1^{m_1}2^{m_2}\cdots)$, an order $\prec$ and a degree-preserving matching.  For any fixed history of partial sewings, \eqref{eq:one-boundary-sum} says that summing all isomorphisms at a degree-$j$ orbit gives $j(1+b)$.  Sum the last sewing first, and then proceed backwards through the order.  This backward induction is valid even though the topology at a later step depends on the earlier choices, because \eqref{eq:one-boundary-sum} is the same for every possible history.  It gives the product of the factors $j(1+b)$ over all matched orbits.  There are $m_j(\mu)!$ matchings for each $j$, so for every fixed order $\prec$,
\[
\sum_{g\in\operatorname{Sew}(C)}\prod_{k=1}^m s_k(g,\prec) =\prod_{j\geq1}m_j(\mu)![j(1+b)]^{m_j(\mu)}.
\]
Averaging over the $m!$ orders in \eqref{eq:order-averaged-sewing-weight} now gives
\begin{equation}\label{eq:gluing-factor}
\sum_{g\in\operatorname{Sew}(C)}s_b(g) =\prod_{j\geq1}m_j(\mu)![j(1+b)]^{m_j(\mu)} =z_\mu(1+b)^{\ellpart(\mu)}.
\end{equation}
Thus the three factors in the gluing constant come from matching equal-degree boundary orbits, cyclic covering isomorphisms, and the summed orientation choices in \eqref{eq:boundary-sewing-rule}.

For fixed $C$, Proposition~\ref{prop:cut-cap} and weighted orbit--stabilizer give
\[
\sum_{[g]\in\operatorname{Sew}(C)/\Aut(C)} \frac{s_b(g)}{|\operatorname{Stab}_{\Aut(C)}(g)|} =\frac1{|\Aut(C)|}\sum_{g\in\operatorname{Sew}(C)}s_b(g).
\]
Multiplying by $w_b(C)$, summing over $C$, and using Lemma~\ref{lem:CD-unlabelled-normalization} gives \eqref{eq:geometric-diagonal-gluing}.  Its right-hand side depends only on $d,r,b$, proving all the stated independence properties.
\end{proof}

For a cover $F$, let $c(F)$ be the number of connected components of its quotient source and set
\begin{equation}\label{eq:normalized-elliptic-weight}
\widetilde W_{b,a}(F) :=\left(\frac{1+b}{2}\right)^{c(F)}w_{b,a}(F).
\end{equation}

\begin{proposition}\label{prop:elliptic-weight-properties}
For a coherent integral MON, $\widetilde W_{b,a}(F)\in\Q_{\geq0}[b]$.  For every coherent MON the weights $w_{b,a}$ are multiplicative under disjoint union.  For an integral MON, their specializations at $b=0$ and $b=1$ are
\begin{equation}\label{eq:elliptic-weight-specializations}
w_{1,a}(F)=1, \qquad w_{0,a}(F)=2^{c(F)}\one_{\{\Sigma\text{ is orientable}\}}.
\end{equation}
Here an empty source has $c(F)=0$ and is orientable.
\end{proposition}

\begin{proof}
Write $F=(C,g)$ relative to the fixed cut.  Every ordered sewing has exactly $c(C)-c(F)$ steps joining distinct components, since only those steps change the number of components.  Each contributes $(1+b)/2$.  Cancelling these factors against \eqref{eq:normalized-elliptic-weight} and the spherical normalization gives
\begin{equation}\label{eq:normalized-weight-product}
\widetilde W_{b,a}(F)=\overline\rho_b(C)\frac1{m!}\sum_{\prec}\prod_{\substack{k:\text{the two circles belong}\\
\text{to the same component at step }k}}s_k(g,\prec).
\end{equation}
Every remaining factor is $1$, $b$ or $(1+b)/2$, proving polynomiality and positivity for an integral MON.

To prove multiplicativity, decompose the final quotient source into two unions of connected components, $F=F_1\sqcup F_2$.  Cutting gives $C=C^{(1)}\sqcup C^{(2)}$, and the sewing data split accordingly; a partial sewing never joins the two unions.  All factors in \eqref{eq:boundary-sewing-rule} depend only on the union in which they occur.  If the two sets of positive-end boundary orbits have sizes $m_1,m_2$, each pair of internal orders has exactly $\binom{m_1+m_2}{m_1}$ interleavings.  Dividing by $(m_1+m_2)!$ therefore factors the order average into the two averages with denominators $m_1!$ and $m_2!$.  Together with \eqref{eq:spherical-weight-multiplicative} this proves $w_{b,a}(F)=w_{b,a}(F_1)w_{b,a}(F_2)$.  The componentwise convention for unused branch labels ensures that these are the weights in the smaller simple-Hurwitz problems themselves.

At $b=1$, both the spherical normalized weight and every sewing factor equal $1$, so $w_{1,a}=1$.  At $b=0$, a nonorientable capped component has zero spherical weight.  If all capped components are orientable but the final source is not, some sewing first makes a component nonorientable and contributes $b$.  For an orientable final source, every internal sewing has factor $1$, every component-joining sewing has factor $1/2$, and the spherical weight is $2^{c(C)}$.  Their product is $2^{c(F)}$.  This proves \eqref{eq:elliptic-weight-specializations}.
\end{proof}

We can now extract covers with connected quotient source.  Define
\begin{equation}\label{eq:connected-elliptic-series}
\mathscr C_b(q;\hbar):=\log\mathscr H_b(q;\hbar), \qquad \Hb^\circ_{1,b}(d,r):=r![q^d\hbar^r]\mathscr C_b(q;\hbar).
\end{equation}
The logarithm is well defined as a formal series in $q$, since $\mathscr H_b(0;\hbar)=1$.

\begin{corollary}\label{cor:connected-quotient-covers}
For every $d\geq1$ and $r\geq0$,
\begin{equation}\label{eq:connected-quotient-count}
\Hb^\circ_{1,b}(d,r)
=\sum_{\substack{[F]\in\mathfrak H_{d,r}(E_T,\iota;\mathbf x)\\
\Sigma\text{ connected}}}\frac{w_{b,a}(F)}{|\Aut(F)|}.
\end{equation}
In particular, the connected weighted count is independent of the admissible cut and of the coherent MON after summation.
\end{corollary}

\begin{proof}
Every cover is a finite disjoint union of covers with connected quotient source.  Each labelled simple branch value belongs to exactly one component, and degrees add.  Proposition~\ref{prop:elliptic-weight-properties} makes the weights multiplicative.  For a fixed multiset of connected isomorphism classes, its automorphism group is the product of the component automorphism groups and the factorials permuting repeated isomorphic components.  Distributing the branch labels gives the usual multinomial factor, absorbed by $\hbar^r/r!$.  The exponential formula therefore identifies the disconnected generating series with the exponential of the right-hand side of \eqref{eq:connected-quotient-count}, summed with $q^d\hbar^r/r!$.  Theorem~\ref{thm:geometric-gluing} identifies that disconnected series with $\mathscr H_b$, proving the corollary.
\end{proof}

At $b=0$, choosing an orientation on each orientable quotient component identifies it with an ordinary oriented cover of $E_T$.  Forgetting these orientations puts a factor $2^{c(F)}$, so that \eqref{eq:elliptic-weight-specializations} recovers ordinary elliptic Hurwitz theory, both connected and disconnected.  At $b=1$ the count is the usual automorphism-weighted count of all real elliptic generalized covers.  A connected quotient source has connected orientation double exactly when it is nonorientable; hence, for $d\geq1$,
\begin{equation}\label{eq:connected-orientation-double-count}
\sum_{\substack{[F]\in\mathfrak H_{d,r}(E_T,\iota;\mathbf x)\\
\widetilde\Sigma\text{ connected}}}\frac1{|\Aut(F)|}
=\Hb^\circ_{1,1}(d,r)-\frac12\Hb^\circ_{1,0}(d,r).
\end{equation}
The factor $1/2$ follows by marking one of the two orientations of a connected orientable quotient source; the resulting objects are ordinary connected oriented covers.

Note that, in analogy to the genus-zero case of \cite{ChapuyDolega22}, the following stronger claim remains open.

\begin{conjecture}\label{conj:cut-independent-elliptic-weight}
There exists a coherent integral MON for the spherical input such that, for every real elliptic generalized cover $F$, the polynomial $\widetilde W_{b,a}(F)$ of \eqref{eq:normalized-elliptic-weight} is independent of the admissible cut $a$.
\end{conjecture}

\subsection{Twisted elliptic Hurwitz numbers}\label{subsec:twisted-elliptic-hurwitz}

We conclude this section with the zonal specialization.  At $b=1$, the preceding real elliptic count recovers the twisted elliptic Hurwitz numbers introduced by Hahn--Markwig \cite{HahnMarkwig26}.  We briefly recall their monodromy convention.  Fix
\[
\tau=(1\ d+1)(2\ d+2)\cdots(d\ 2d)\in\mathfrak S_{2d}, \qquad B_d:=C_{\mathfrak S_{2d}}(\tau), \qquad |B_d|=(2d)!!=2^d d!,
\]
and let $B_d^{\sim}$ be the set of $\sigma\in\mathfrak S_{2d}$ satisfying
$\tau\sigma\tau=\sigma^{-1}$ and having no cycle fixed by $\tau$.  The disconnected invariant $\widetilde h_{d,r+1}^{\bullet}$ of \cite[Definition~3]{HahnMarkwig26} is $(2d)!!^{-1}$ times the number of tuples
\[
(\sigma,\eta_1,\ldots,\eta_r,\beta), \qquad \sigma\in B_d^{\sim},\quad \beta\in B_d,
\]
where each $\eta_s=(i_sj_s)$ is a transposition with $j_s\neq\tau(i_s)$ and
\begin{equation}\label{eq:twisted-monodromy-relation}
\eta_1\cdots\eta_r\sigma (\tau\eta_r\tau)\cdots(\tau\eta_1\tau) =\beta\sigma\beta^{-1}.
\end{equation}
The connected invariant imposes transitivity of the group generated by the entries of this tuple and their $\tau$-conjugates.

Geometrically, $\eta_s$ and $\tau\eta_s\tau$ are the monodromies of a pair of reflected simple branch points, $\sigma$ records the invariant seam, and $\beta$ closes the transverse direction; this convention is illustrated in Figure~\ref{fig:twisted-elliptic-monodromy}.

\begin{figure}[h!]
\centering
\includegraphics[width=0.5\textwidth]{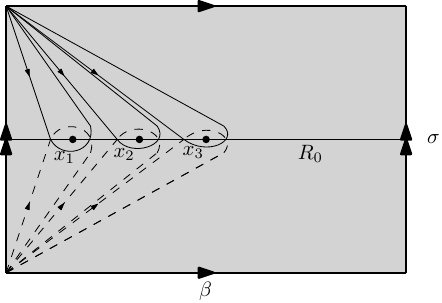}
\caption{Hahn--Markwig monodromy convention on a fundamental polygon of the real elliptic target $(E_T,\iota)$.  The simple branch values $x_1,x_2,x_3$ lie on the real oval $R_0$; the indicated loops encode the branch monodromies together with the seam and transverse monodromies.}
\label{fig:twisted-elliptic-monodromy}
\end{figure}

The following matching calculation identifies this monodromy count with the zonal elliptic $b$-Hurwitz trace.

\begin{lemma}\label{lem:twisted-matching-trace}
For every $d,r\geq0$,
\begin{equation}\label{eq:twisted-matching-trace}
\widetilde h_{d,r+1}^{\bullet} =\Tr_{\Lambda_d}\bigl((2\Dcal_2)^r\bigr) =2^r\Hb_{1,1}(d,r).
\end{equation}
For $d=0$, only the empty tuple with $r=0$ contributes.
\end{lemma}

\begin{proof}
Let $\mathcal M_d$ be the set of perfect matchings of $\{1,\ldots,2d\}$ and let $\mathscr V_d$ have basis $(e_\kappa)_{\kappa\in\mathcal M_d}$.  Writing $U_\gamma e_\kappa=e_{\gamma\kappa\gamma^{-1}}$, set
\[
\mathsf T_d:=\sum_{\substack{\eta=(ij)\\j\ne\tau(i)}}U_\eta,
\qquad
\mathsf P_d:=\frac1{|B_d|}\sum_{\beta\in B_d}U_\beta.
\]
The admissible transpositions are invariant under conjugation by $B_d$, so $\mathsf T_d$ commutes with $\mathsf P_d$, the projection onto $\mathscr V_d^{B_d}$.  The map $\sigma\mapsto\kappa=\sigma\tau$ is a bijection $B_d^{\sim}\to\mathcal M_d$: superposing the edges of $\tau$ and $\kappa$ produces alternating cycles whose half-lengths are precisely the paired cycle lengths of $\sigma$.  Multiplying \eqref{eq:twisted-monodromy-relation} by $\tau$ therefore turns it into
\[
\eta_1\cdots\eta_r\kappa\eta_r\cdots\eta_1 =\beta\kappa\beta^{-1},
\]
whence
\[
\widetilde h_{d,r+1}^{\bullet} =\Tr_{\mathscr V_d}(\mathsf P_d\mathsf T_d^r) =\Tr_{\mathscr V_d^{B_d}}(\mathsf T_d^r).
\]
The $B_d$-orbits of matchings are indexed by partitions $\mu\vdash d$, through the half-lengths of their alternating cycles.  If $\mathcal O_\mu$ is the orbit of type $\mu$, set $v_\mu=|\mathcal O_\mu|^{-1}\sum_{\kappa\in\mathcal O_\mu}e_\kappa$; then the coefficient from $v_\mu$ to $v_\nu$ is the number of admissible transpositions taking a fixed matching of type $\mu$ to type $\nu$.  On $\mathscr V_d^{B_d}$, the $d$ transpositions forming the edges of $\tau$ act trivially, hence $\mathsf T_d$ is the restriction of the sum over all transpositions minus $dI$.  For a fixed matching $\kappa$, the $d$ transpositions forming its edges fix $\kappa$, so this subtraction cancels precisely those diagonal terms.  What remains selects two distinct $\kappa$-edges and reconnects their four endpoints, each of the two reconnections being induced by exactly two transpositions.  Consequently the matrix is exactly twice the cut, join and diagonal matrix of \eqref{eq:Dalpha} at $\alpha=2$.  For example, joining two distinct alternating cycles of half-lengths $i<j$ has multiplicity $4ijm_i(\mu)m_j(\mu)$, twice the corresponding coefficient of $\Dcal_2$.  The remaining multiplicities are $4i^2\binom{m_i(\mu)}{2}$ for joining equal cycles, $2km_k(\mu)$ for a split $k=i+j$ with $i<j$, $2im_{2i}(\mu)$ for an equal split, and $\sum_k k(k-1)m_k(\mu)$ for type-preserving reconnections; these are likewise twice the corresponding coefficients of $\Dcal_2$.  Thus $v_\mu\mapsto p_\mu$ identifies $\mathsf T_d|_{\mathscr V_d^{B_d}}$ with $2\Dcal_2$, and taking traces proves the claim.
\end{proof}

For the connected twisted invariant there is one extra point: transitivity is imposed before adjoining the involution $\tau$.  Its effect is summarized by the following exponential decomposition.

\begin{lemma}\label{lem:twisted-connected-series}
Let
\[
\mathscr U(q;\hbar):=\sum_{d\geq1}\sum_{r\geq0} \widetilde h_{d,r+1}q^d\frac{\hbar^r}{r!}.
\]
Then
\begin{equation}\label{eq:twisted-connected-series}
\mathscr U(q;\hbar) =\mathscr C_1(q;2\hbar)-\frac12\mathscr C_0(q;2\hbar).
\end{equation}
\end{lemma}

\begin{proof}
For a twisted elliptic monodromy tuple, let $G$ be the subgroup generated by $\sigma,\beta$, the $\eta_s$, and their $\tau$-conjugates.  Since $\tau G\tau=G$, the involution $\tau$ permutes the $G$-orbits.  An orbit of $\langle G,\tau\rangle$ is therefore either a single $\tau$-stable $G$-orbit, giving a connected twisted elliptic component, or a pair of $G$-orbits exchanged by $\tau$.  After marking one orbit in such a pair, restriction gives an ordinary connected elliptic Hurwitz monodromy tuple, while each of the $r$ branch labels records on which member of the pair its transposition lies.  Thus a paired component contributes $2^r$ times the ordinary connected elliptic Hurwitz weight, and forgetting the marked member divides its groupoid cardinality by $2$.

The exponential formula for the decomposition into $\langle G,\tau\rangle$-orbits gives
\[
\sum_{d,r\geq0}\widetilde h_{d,r+1}^{\bullet}q^d\frac{\hbar^r}{r!} =\exp\left(\mathscr U(q;\hbar)+\frac12\mathscr C_0(q;2\hbar)\right).
\]
By Lemma~\ref{lem:twisted-matching-trace}, the left-hand side is $\mathscr H_1(q;2\hbar)$.  Taking logarithms proves \eqref{eq:twisted-connected-series}.  For completeness, on a paired orbit choose one member $\Omega$.  For each $s$, exactly one of $\eta_s$ and $\tau\eta_s\tau$ acts nontrivially on $\Omega$; transport the other factor back by $\tau$ and conjugate it through $\sigma$.  If this changes the branch order, restore it by the Hurwitz move $(u,v)\mapsto(v,vuv)$, which preserves both the product and the generated subgroup.  Conversely, fixing the inverse braid for each binary side pattern reverses the construction and extends the seam and transverse monodromies by $\tau\sigma\tau=\sigma^{-1}$ and $\tau\beta\tau=\beta$.  Hence a marked ordinary component has exactly $2^r$ lifts, while forgetting the marking gives the factor $1/2$.
\end{proof}

\begin{proposition}\label{prop:twisted-real-monodromy}
For $d\geq0$ and $r\geq0$,
\begin{equation}\label{eq:twisted-comparison}
\widetilde h_{d,r+1}^{\bullet} =2^r\Hb_{1,1}(d,r) =2^r\sum_{[F]\in\mathfrak H_{d,r}(E_T,\iota;\mathbf x)}\frac1{|\Aut(F)|}.
\end{equation}
For $d\geq1$,
\begin{equation}\label{eq:twisted-connected-geometric}
\widetilde h_{d,r+1}=2^r\sum_{\substack{[F]\in\mathfrak H_{d,r}(E_T,\iota;\mathbf x)\\
\widetilde\Sigma\text{ connected}}}\frac1{|\Aut(F)|},
\end{equation}
and the connected orientation double $\widetilde\Sigma$ has genus $r+1$.
\end{proposition}

\begin{proof}
At $b=1$, Proposition~\ref{prop:elliptic-weight-properties} gives $w_{1,a}(F)=1$.  Theorem~\ref{thm:geometric-gluing} together with Lemma~\ref{lem:twisted-matching-trace} yields \eqref{eq:twisted-comparison}.  Taking coefficients in \eqref{eq:twisted-connected-series} gives
\[
\widetilde h_{d,r+1} =2^r\left(\Hb^\circ_{1,1}(d,r)-\frac12\Hb^\circ_{1,0}(d,r)\right),
\]
which is \eqref{eq:twisted-connected-geometric} by \eqref{eq:connected-orientation-double-count}.  Equation~\eqref{eq:real-RH} then gives $g=r+1$.

Finally, the Riemann existence theorem equips the oriented cover $\widetilde F:\widetilde\Sigma\to E_T$ with the unique complex structure making $\widetilde F$ holomorphic; equivariance makes $\jmath$ antiholomorphic.  Thus these are geometric counts of real covers in the sense proposed in \cite[Remark~5]{HahnMarkwig26}.  The factor $2^r$ comes from the matching-operator normalization, not from independently choosing the reflected transposition in a fixed tuple.
\end{proof}

\section{Elliptic $b$-Hurwitz expansion of the Jack heat trace}\label{sec:heat-trace-expansion}

We now combine Sections~\ref{sec:jack-heat-trace} and~\ref{sec:elliptic-b-hurwitz}, specializing to $\alpha=1+b\geq1$, where the coefficients have their enumerative interpretation.  The exact finite-rank trace factorizes into two length-truncated elliptic $b$-Hurwitz sectors coupled by the determinant charge; removing the cutoffs gives the formal amplitude whose fixed truncations reproduce the large-rank expansion.

\subsection{Exact finite-rank factorization}\label{subsec:finite-rank-factorization}
In the stable coordinates of Section~\ref{subsec:stable-representation}, the two partitions $\mu$ and $\nu$ enter symmetrically, while the determinant charge tilts their degree weights in opposite directions.  We call the corresponding factors the chiral and antichiral sectors; their finite-rank length cutoffs make the series below convergent, whereas the unrestricted stable sectors are formal in the content variable.

Set $x_N=(N+b)^{-1}$ and retain the notation $q_t=e^{-t/2}$ from Section~\ref{subsec:probabilistic-representation}.  Abbreviate the stable correction at Jack parameter $1+b$ by
\begin{equation}\label{eq:Fb}
F_b(\mu,\nu,m) :=\cA_{1+b}(\mu)+\cA_{1+b}(\nu) +(1+b)m(|\mu|-|\nu|) -\frac{b(1+b)}{2}m^2.
\end{equation}
Throughout this section, $M$ has the discrete-Gaussian law \eqref{eq:discrete-Gaussian} with $\alpha=1+b$; when used jointly with $\mu,\nu$, the latter are independent $q_t$-uniform partitions and are independent of $M$.
For $L\geq0$, introduce the length-truncated counterpart of \eqref{eq:elliptic-master-spectral},
\begin{equation}\label{eq:finite-Hurwitz-sector}
\mathscr H_{b,L}(z;\hbar):=\sum_{\substack{\lambda\in\Pcal\\ \ellpart(\lambda)\leq L}}z^{|\lambda|}e^{\hbar\cA_{1+b}(\lambda)}.
\end{equation}
Thus $L$ records the finite-rank cutoff, while formally removing it gives $\mathscr H_b(z;\hbar)$.  For $L=0$ the sum consists only of the empty partition.

For fixed $b\geq0$, $L\geq0$, $z>0$, and $x>0$, the specialization $\mathscr H_{b,L}(z;-x)$ converges absolutely.  Indeed, for $d=|\lambda|$ and $\ellpart(\lambda)\leq L$, \eqref{eq:K-sharp-lower} and $\cA_{1+b}(\lambda)\geq K(\lambda)$ give
\[
\cA_{1+b}(\lambda)\geq \frac{d^2}{2L}-\frac{Ld}{2}
\]
when $L\geq1$; hence the degree-$d$ contribution is bounded by
$p(d)z^d\exp(-xd^2/(2L)+xLd/2)$ and is summable, where $p(d)$ is the partition number introduced in Section~\ref{subsec:stable-representation}.  The case $L=0$ is trivial.  The main result of this subsection is that, conditioned on the determinant charge $M$, the finite-rank Jack-deformed central heat trace is the product of a chiral and an antichiral length-truncated Jack/Hurwitz sector.

\begin{proposition}\label{prop:finite-N-chiral-factorization}
For every $N\geq1$, $b\geq0$, and $t>0$,
\begin{equation}\label{eq:finite-N-chiral-factorization}
\begin{aligned}
Z_{1+b,N}(t)={}&\Theta_{1+b}(t)\E_M\Big[e^{\frac12 b(1+b)t x_NM^2}\\
&\qquad\qquad\times\mathscr H_{b,A_N}\left(q_t e^{-(1+b)Mt x_N};-t x_N\right)\mathscr H_{b,B_N}\left(q_t e^{(1+b)Mt x_N};-t x_N\right)\Big].
\end{aligned}
\end{equation}
\end{proposition}

\begin{proof}
Specializing the stable-coordinate bijection $\lambda_N$ and Proposition~\ref{prop:stable-decomposition} at Jack parameter $1+b$ gives
\begin{align*}
Z_{1+b,N}(t)&=\sum_{m\in\Z}\sum_{\substack{\ellpart(\mu)\leq A_N\\ \ellpart(\nu)\leq B_N}}q_t^{|\mu|+|\nu|}e^{-\frac t2(1+b)m^2}e^{-t x_NF_b(\mu,\nu,m)}.
\end{align*}
By \eqref{eq:Fb},
\begin{align*}
e^{-t x_NF_b(\mu,\nu,m)}={}&e^{-t x_N\cA_{1+b}(\mu)}e^{-(1+b)mt x_N|\mu|}e^{-t x_N\cA_{1+b}(\nu)}e^{(1+b)mt x_N|\nu|}e^{\frac12b(1+b)t x_Nm^2}.
\end{align*}
The sums over $\mu$ and $\nu$ therefore factor, and \eqref{eq:discrete-Gaussian} converts the remaining charge sum into the expectation in \eqref{eq:finite-N-chiral-factorization}.  The individual partition sums converge by the convergence estimate preceding the proposition.
\end{proof}

\subsection{Chiral--antichiral generating series and controlled expansion}\label{subsec:controlled-expansion}
We now remove the length cutoffs in Proposition~\ref{prop:finite-N-chiral-factorization} and work formally in the content variable.  The determinant charge still shifts the two degree variables in opposite directions, while $x$ records the inverse-rank parameter $1/(N+b)$.

Let $\mathfrak{D}_q:=q\frac{\dd}{\dd q}$.  For a formal variable $u$, we interpret a shifted argument by
\begin{equation}\label{eq:formal-q-shift}
\mathscr H_b(qe^u;\hbar) :=e^{u\mathfrak{D}_q}\mathscr H_b(q;\hbar).
\end{equation}
Thus a partition of degree $d$ acquires the factor $e^{ud}$; no analytic meaning of the shifted series away from $u=0$ is required.  For fixed $t>0$ and $b\geq0$, define the formal chiral--antichiral generating series
\begin{equation}\label{eq:Hurwitz-master-Z}
\begin{aligned}
\mathcal Z_{b,t}(x):={}&\Theta_{1+b}(t)\E_M\Big[e^{\frac12 b(1+b)txM^2}\\ &\qquad\qquad\times\mathscr H_b(q_t e^{-(1+b)tMx};-tx)\mathscr H_b(q_t e^{(1+b)tMx};-tx)\Big].
\end{aligned}
\end{equation}
The finite-rank factorization \eqref{eq:finite-N-chiral-factorization} and the generating series \eqref{eq:Hurwitz-master-Z} have the same form, the former being obtained by imposing the two length cutoffs and specializing $x=x_N$.

If $\mu,\nu$ are independent $q_t$-uniform partitions, independent of $M$, then \eqref{eq:elliptic-master-spectral} and \eqref{eq:Fb} give the equivalent joint-expectation form
\begin{equation}\label{eq:Hurwitz-master-expectation}
\mathcal Z_{b,t}(x) =\frac{\Theta_{1+b}(t)}{\phi(q_t)^2} \E\left[e^{-txF_b(\mu,\nu,M)}\right],
\end{equation}
where the expectation is understood coefficientwise as a formal power series in $x$.  Every coefficient is therefore well defined and
\begin{equation}\label{eq:master-coefficients}
[x^r]\mathcal Z_{b,t}(x) =\frac{\Theta_{1+b}(t)}{\phi(q_t)^2} \frac{(-t)^r}{r!} \E\left[F_b(\mu,\nu,M)^r\right], \qquad r\geq0.
\end{equation}
The required moments are finite by the elementary bounds \eqref{eq:q-tail} and \eqref{eq:gaussian-moments}.

\begin{proposition}\label{prop:master-zero-radius}
Write $a_m(t,b)=[x^m]\mathcal Z_{b,t}(x)$.  For every compact $K\subset(0,\infty)\times[0,\infty)$ there are constants $A_K,B_K,c_K>0$ and an integer $r_K$ such that, for $(t,b)\in K$,
\begin{equation}\label{eq:gevrey-coefficient-bounds}
|a_m(t,b)|\leq A_K B_K^m m!\quad(m\geq0), \qquad a_{2r}(t,b)\geq c_K^{2r}(2r)!\quad(r\geq r_K).
\end{equation}
In particular, for every fixed $t>0$ and $b\geq0$, the coefficient sequence has exact Gevrey order one and $\mathcal Z_{b,t}(x)$ has radius of convergence zero.
\end{proposition}

\begin{proof}
Fix $K$ and choose $0<t_0\leq t\leq t_1$ and $0\leq b\leq b_1$ on $K$.  For $R=1+|\mu|+|\nu|+M^2$, the product partition and discrete-Gaussian laws have a uniformly finite exponential moment $\sup_K\E[e^{\delta R}]<\infty$ for sufficiently small $\delta>0$.  This follows from Euler's partition product and the Gaussian sum, since $\delta<t_0/2$ leaves a strictly negative coefficient of each size and of $M^2$.  The content bound gives $|F_b|\leq C_K R^2$.  Using $R^{2m}\leq(2m)!\delta^{-2m}e^{\delta R}$ in \eqref{eq:master-coefficients}, and bounding the prefactor $\Theta_{1+b}(t)/\phi(q_t)^2$ on $K$, we obtain
\[
|a_m(t,b)|\leq A_K C_K^m\frac{(2m)!}{m!} \leq A_K(4C_K)^m m!.
\]
This proves the upper bound after renaming the constant.

By \eqref{eq:master-coefficients}, the even coefficients are nonnegative.  Restricting the expectation defining $a_{2r}$ to the event $\{M=0,\ \nu=\varnothing,\ \mu=(d)\}$ and using
\[
\cA_{1+b}((d))=\frac{1+b}{2}d(d-1),
\]
gives, for every $d\geq2$,
\begin{equation}\label{eq:master-zero-radius-lower}
a_{2r}\geq \frac{q_t^d}{(2r)!} \left(\frac{t(1+b)}{2}d(d-1)\right)^{2r}.
\end{equation}
Indeed, the factors $\phi(q_t)^2\Theta_{1+b}(t)^{-1}$ coming from the probability of this event cancel the prefactor in \eqref{eq:master-coefficients}.  For all sufficiently large $r$, choose $d_r=\lfloor 8r/t\rfloor$.  Then $q_t^{d_r}\geq e^{-4r}$ and $d_r(d_r-1)\geq 8r^2/t^2$.  Since $(2r)!\leq(2r)^{2r}$, \eqref{eq:master-zero-radius-lower} yields
\[
a_{2r}^{1/(2r)} \geq e^{-2}\frac{2(1+b)}{t} r,
\]
for all sufficiently large $r$, with the threshold uniform for $t\in[t_0,t_1]$.  Hence
\[
a_{2r}(t,b)\geq\left(\frac{2e^{-2}r}{t_1}\right)^{2r} \geq\left(\frac{e^{-2}}{t_1}\right)^{2r}(2r)!,
\]
which proves the lower bound.  Stirling's formula excludes every coefficient bound of Gevrey order $s<1$.  The same lower bound gives $\limsup_m|a_m|^{1/m}=\infty$, so the radius of convergence is zero.
\end{proof}

The proposition concerns growth of the coefficients.  The controlled expansion below is uniform in the parameters at every fixed truncation order; it does not assert a remainder estimate uniform in that order or a Borel summation theorem.

\begin{theorem}\label{thm:controlled-chiral-antichiral}
For every compact set $K\subset(0,\infty)\times[0,\infty)$ and every $p\in\Z_{\geq0}$, there exist $C_{K,p}<\infty$ and $N_{K,p}\geq1$ such that, for every $N\geq N_{K,p}$,
\begin{equation}\label{eq:controlled-master-uniform}
\sup_{(t,b)\in K} (N+b)^{p+1} \left| Z_{1+b,N}(t) -\bigl[\mathcal Z_{b,t}\bigr]_{\leq p}\left(\frac{1}{N+b}\right) \right| \leq C_{K,p}.
\end{equation}
\end{theorem}

\begin{proof}
By the coefficient identity \eqref{eq:master-coefficients},
\begin{align*}
\bigl[\mathcal Z_{b,t}\bigr]_{\leq p}\left(\frac{1}{N+b}\right)&=\frac{\Theta_{1+b}(t)}{\phi(q_t)^2}\sum_{r=0}^p\frac{(-t)^r}{r!(N+b)^r}\E\left[F_b(\mu,\nu,M)^r\right].
\end{align*}
This is exactly Theorem~\ref{thm:heat-expansion} specialized to Jack parameter $1+b$.  The remainder and its local uniformity therefore give \eqref{eq:controlled-master-expansion}--\eqref{eq:controlled-master-uniform}.
\end{proof}

\bibliographystyle{alpha}
\bibliography{Elliptic-b-Hurwitz}

\end{document}